\documentclass[12pt,a4paper,reqno]{amsart}
\usepackage[margin=1.3in]{geometry}

\usepackage{graphicx}
\usepackage{amsmath}
\usepackage{amscd}
\usepackage{amsfonts}
\usepackage{amssymb}
\usepackage{amsthm}
\usepackage[dvipsnames]{xcolor}
\usepackage{mathtools}
\usepackage{todonotes}
\usepackage{tikz-cd} 

\usepackage[hypertexnames=false, colorlinks, citecolor=red, linkcolor=blue, urlcolor=red, bookmarksopen, backref=page]{hyperref}

\usepackage{enumitem}
\newcounter{dummy}
\makeatletter
\newcommand\myitem[1][]{\item[#1]\refstepcounter{dummy}\def\@currentlabel{#1}}
\makeatother

\usepackage[depth=3]{bookmark}
\let\oldtocsection=\tocsection
\let\oldtocsubsection=\tocsubsection
\let\oldtocsubsubsection=\tocsubsubsection
\renewcommand{\tocsection}[2]{\hspace{0em}\vspace{0.5mm}\oldtocsection{#1}{#2}\vspace{0.5mm}}
\renewcommand{\tocsubsection}[2]{\hspace{2em}\vspace{0.25mm}\oldtocsubsection{#1}{#2}\vspace{0.25mm}}
\renewcommand{\tocsubsubsection}[2]{\hspace{2em}\oldtocsubsubsection{#1}{#2}}

\numberwithin{equation}{section}

\newtheorem{theorem}{Theorem}[section]
\newtheorem{lemma}[theorem]{Lemma}
\newtheorem{proposition}[theorem]{Proposition}

\theoremstyle{definition}
\newtheorem{example}[theorem]{Example}
\newtheorem{remark}[theorem]{Remark}
\newtheorem{definition}[theorem]{Definition}

\renewcommand{\qedsymbol}{{\vrule height5pt width5pt depth1pt}}

\newcommand{\be}{\begin{equation}}
	\newcommand{\ee}{\end{equation}}
\newcommand{\bes}{\begin{equation*}}
	\newcommand{\ees}{\end{equation*}}

\newcommand{\cA}{\mathcal{A}}
\newcommand{\cB}{\mathcal{B}}

\newcommand{\cD}{\mathcal{D}}

\newcommand{\cH}{\mathcal{H}}

\newcommand{\cK}{\mathcal{K}}

\newcommand{\cM}{\mathcal{M}}

\newcommand{\cS}{\mathcal{S}}

\newcommand{\cU}{\mathcal{U}}
\newcommand{\cV}{\mathcal{V}}

\newcommand{\cZ}{\mathcal{Z}}

\newcommand{\bB}{\mathbb{B}}
\newcommand{\bC}{\mathbb{C}}
\newcommand{\bD}{\mathbb{D}}

\newcommand{\bF}{\mathbb{F}}

\newcommand{\bM}{\mathbb{M}}
\newcommand{\bN}{\mathbb{N}}

\newcommand{\bS}{\mathbb{S}}
\newcommand{\bT}{\mathbb{T}}
\newcommand{\bZ}{\mathbb{Z}}

\newcommand{\lip}{\langle}
\newcommand{\rip}{\rangle}
\newcommand{\ip}[1]{\lip #1 \rip}

\newcommand{\bmat}[1]{\begin{bmatrix}
		#1
\end{bmatrix}}

\newcommand{\bsmallmat}[1]{\begin{bsmallmatrix}
		#1
\end{bsmallmatrix}}

\newcommand{\diag}{\operatorname{diag}}
\newcommand{\col}{\operatorname{col}}
\newcommand{\row}{\operatorname{row}}
\newcommand{\iso}{\operatorname{iso}}
\newcommand{\coiso}{\operatorname{coiso}}
\newcommand{\uni}{\operatorname{uni}}

\newcommand{\nc}{\operatorname{nc}}

\newcommand{\fB}{{\mathfrak{B}}}
\newcommand{\fC}{{\mathfrak{C}}}
\newcommand{\fD}{{\mathfrak{D}}}

\newcommand{\foral}{\text{ for all }}
\newcommand{\qand}{\quad\text{and}\quad}

\newcommand{\AND}{\text{ and }}

\newtheorem{mainthm}{Theorem}

\usepackage{etoolbox}
\apptocmd{\sloppy}{\hbadness 10000\relax}{}{}
\apptocmd{\sloppy}{\vbadness 10000\relax}{}{}

\begin{document}

\nobreakdash
	
\title[Zeros of Holomorphic Functions as Spectral Data]{Zeros of Holomorphic Functions in Commuting and Non-commuting Variables as Spectral Data}
		
\author{Poornendu Kumar}
\address{Department of Mathematics, University of Manitoba, Winnipeg, Canada}
\email{Poornendu.Kumar@umanitoba.ca}
	
\author{Jeet Sampat}
\address{Department of Mathematics, University of Manitoba, Winnipeg, Canada}
\email{Jeet.Sampat@umanitoba.ca}

\subjclass[2020]{Primary 32A60, 46L52, 47A10, 47A48}

\thanks{This project is supported in part by the Pacific Institute for the Mathematical Sciences}
	
\begin{abstract}
    We characterize the zero sets of functions in the Schur--Agler class over the unit polydisk as well as functions in the unit ball of the multiplier algebra of the Drury--Arveson space via operators associated with a unitary realization formula for these functions. To this end, new notions of `eigenvalues' for tuples of operators are introduced, where the eigenvalues depend on the operator space structure of the ambient domain. Several examples showcasing the properties of these eigenvalues and the zero sets of rational inner functions in the Schur--Agler class are also presented.
    
    We further generalize this result to a large class of non-commuting (NC) holomorphic functions whose ambient domain is given by the unit ball of a matrix of linear polynomials. This includes the NC counterparts of the unit polydisk and the Euclidean unit ball. We also show for functions in the Schur--Agler class over NC matrix unit balls that their zeros along the topological boundary are contained in an appropriately defined `approximate point spectrum' of the associated realization operator, and so are points along the Shilov boundary where the boundary values are not isometric/coisometric. This, in-turn, provides an identical result for the commutative case.

    \vspace{2mm}
    
    \noindent \textbf{Keywords.} Eigenvalues, Matrix unit balls, Non-commutative functions, Realization formulae, Schur--Agler class, Zeros of holomorphic functions.
\end{abstract}
	
\maketitle

\section{Introduction}

\subsection{Motivation}\label{subsec:intro.motivation}
    
Let $\mathcal{S}(\Omega)$ be the collection of all  holomorphic maps from a bounded domain $\Omega \in \bC^d$ ($d \geq 1$) into the \emph{unit disk} $\bD$. One of the first things we learn in a graduate course in complex analysis is that the zero set $\cZ_\bD(f)$ (counting multiplicities) of any given $f \in \cS(\bD)$ satisfies the \emph{Blaschke condition}, which in-turn implies that $\cZ_\bD(f)$ is at most countable. Conversely, given a Blaschke sequence $\underline{\lambda} = \{\lambda_k\}_{k \in \bN} \subset \bD$, one can construct a Blaschke product $B_{\underline{\lambda}} \in \cS(\bD)$ whose zero set is precisely $\underline{\lambda}$. Now, if $f \in \cS(\bD)$ has $\underline{\lambda}$ as its zero set, then we can we factor out the zeros of $f$ via $f = B_{\underline{\lambda}} g$ for some non-vanishing function $g \in \cS(\bD)$. This is a crucial step in analyzing the structure of functions in $\cS(\bD)$ as showcased by Smirnov's factorization theorem (see \cite[Theorem 2.8]{Duren-Book}) and is a cornerstone in the theory of Hardy spaces. Moving beyond $\cS(\bD)$, however, one quickly realizes that the characterization of zero sets is quite complicated. We therefore take the following question as our prime motivation.
\begin{quote}
\textbf{Question.} Can we characterize the zero sets of functions in $\mathcal{S}(\Omega)$?
\end{quote}
In general, this is a difficult problem. There are two approaches one can take: either impose certain geometric conditions on the domains, or work with subalgebras of $\cS(\Omega)$ instead that possess additional function theoretic structure.

We take the latter approach and build a correspondence between zero sets of a subclass of $\cS(\Omega)$ through spectral data of operators associated with this class. Let $\Omega = \bD$ as before but, this time, replace the \emph{Schur class} $\cS(\bD)$ with $\cM(\cD)_1$ -- the unit ball of the multiplier algebra of the \emph{Dirichlet space} $\cD$, which sits comfortably inside the Schur class but is not equal to it. It then follows from the previous discussion that the zero set of every function in $\cM(\cD)_1$ is a Blaschke sequence, but it is not true that every Blaschke sequence corresponds to the zero set of some $f \in \cM(\cD)_1$ (see \cite[Chapter 4]{El-Fallah-Kellay-Mashreghi-Ransford-Dirichlet-Book}). In fact, there is no known necessary as well as sufficient condition for zero sets of functions in $\cM(\cD)_1$. A similar situation arises if we take the \emph{unit polydisk}
\begin{equation*}
    \bD^d := \{ z = (z_1, \dots, z_d) \in \bC^d : |z_j| < 1, \; \forall 1 \leq j \leq d\}.
\end{equation*}
One quickly realizes that analogues of results in the $d = 1$ case hold only for those $f \in \cS(\bD^d)$ such that $\log |f^\dagger| = \operatorname{Re} g$ for some $g \in \operatorname{Hol}(\bD^d)$. For such an $f$, the zero set of $f$ is the same as the zero set of an \emph{inner function} $\varphi \in \cS(\bD^d)$, i.e., $|\varphi^\dagger| = 1$ a.e. (see \cite[Theorem 5.4.1 and 5.4.5]{Rudin-Book-Polydisk}). Here, $\varphi^\dagger$ represents the `boundary' function of $\varphi$ along the \emph{distinguished boundary}
$$\bT^d := \{z \in \bC^d : |z_j| = 1, \, \forall 1 \leq j \leq d\},$$
obtained by taking \emph{radial limits} $\lim_{r \to 1} \varphi(r \lambda)$ for Lebesgue a.e. $\lambda \in \bT^d$. However, characterizing zeros of inner functions is also challenging.

Let us take a step back to the case of $\cS(\bD)$. A remarkable result in function-theoretic operator theory is the existence of a realization formula for $\cS(\bD)$, which states that every $f \in \cS(\bD)$ admits a formula
\begin{equation} \label{eqn:f.realization}
    f(z) = A + zB(I - zD)^{-1}C \foral z \in \bD,
\end{equation}
where
\begin{equation}\label{eqn:block.matrix.rep.for.V}
    V := \bmat{A & B \\	C & D} : \bC \oplus \cH\to \bC \oplus \cH
\end{equation}
can be chosen to be a unitary operator for some Hilbert space $\mathcal{H}$. It is not surprising that any function expressed in the form \eqref{eqn:f.realization} is holomorphic on $\bD$ and belongs to $\cS(\bD)$. This is referred to as a \emph{transfer function realization} -- a term originating from engineering -- or, more simply, as a \emph{realization}. The study of transfer functions extends beyond just understanding zeros, and it is difficult to compile a complete set of resources surrounding this topic. We therefore refer the reader to the books \cite{Agler-McCarthy-Young-Book, Ball-Gohberg-Rodman-book, Rosenbrock-book} for relevant information. In particular, the relationship between zeros of certain $f$ and the realization operator $D$ are known to exist in other settings but we are interested in the Schur class. To be consistent with the rest of this paper, we mention the precise statement we wish the reader to keep in mind throughout this discussion.

\begin{theorem}\label{thm:disk.case}
    If $f \in \cS(\bD)$ and $V = \bsmallmat{A & B \\ C & D}$ are as in \eqref{eqn:f.realization} and \eqref{eqn:block.matrix.rep.for.V}, then
    \begin{equation*}
        \cZ_\bD(f) = \sigma_p(D^*) \cap \bD.
    \end{equation*}
\end{theorem}
Here, $\sigma_p(D^*)$ denotes the set of eigenvalues of the operator $D^* \in \cB(\cH)$, and we refer to $D^*$ as the \emph{associated operator} for $f$.

\subsection{Realization formulae in several variables}\label{subsec:intro.realization.in.several.vars}
 The realization formula is a fundamental tool, having inspired a broad spectrum of results across diverse areas of mathematics, including analysis of several complex variables, multivariable operator theory, function theory, operator algebras, and engineering. To mention just a few instances: it facilitates the derivation of the Pick–Nevanlinna interpolation problem \cite{Agler-McCarthy-Young-Book}, proves the commutant lifting theorem \cite{Sarason-TAMS-1967}, yields the Toeplitz corona theorem \cite{Ball-Trent-Vinnikov}, and establishes the Carathéodory approximation result on the disk as well as the bidisk \cite{Alpay-Bhattacharyya-Jindal-Kumar-BLMS}. It has also been applied to factorization results \cite{Bhowmik-Kumar, Brodskiu-1978, Ramlal-Sarkar}, the invariant subspace problem \cite{Brodskiu-1978, Sz.-Nagy-Foias-Book}, operator algebras \cite{Mittal-Paulsen}, the extension of Herglotz integral representation \cite{Agler-McCarthy-Young-Book,Bhowmik-Kumar-II, Pascoe-Her}, and the extension problem for holomorphic maps \cite{Agler-McCarthy-Young-Book}. Beyond pure mathematics, its utility extends to applications such as signal processing \cite{Ball-IEOT-1987}, electrical engineering \cite{Helton-IUMJ-1972}, and linear image processing \cite{Roesser}. In this article, we shall employ the realization formula for the purpose of studying zeros. 
 \subsubsection*{\texorpdfstring{\textbf{For commutative functions}}{For commutative functions}}Two particularly prominent settings are the unit polydisk $\bD^d$ as introduced earlier, and the \emph{Euclidean unit ball} $$\bB_d := \left\{ z \in \bC^d : \sum_{j = 1}^d |z_j|^2 < 1 \right\}.$$

Agler \cite{Agler} made a notable contribution by generalizing the realization formula to $\bD^d$. On the \emph{bidisk} $\bD^2$, the realization formula holds for all functions in $\cS(\bD^2)$, however, for $d > 2$, Agler identified a subclass of $\cS(\bD^d)$ consisting of functions satisfying a von Neumann type inequality for which the realization formula remains valid. This subclass is now known as the {\em Schur--Agler class} on $\bD^d$, which we denote by $\cS\cA(\bD^d)$. Agler proved that $f \in \mathcal{SA}(\bD^d)$ if and only if we can find auxillary Hilbert spaces $\cH_1, \dots, \cH_d$ and a unitary colligation
\begin{equation*}
    V = \bmat{A & B \\ C & D} : \bmat{\bC \\ \oplus_{j = 1}^d \cH_j} \to \bmat{\bC \\ \oplus_{j = 1}^d \cH_j}
\end{equation*}
such that $f$ can be written as
\begin{align}\label{eqn:intro.realization.form.polydisk}
    f(z) = A + B \Delta(z) (I - D \Delta(z))^{-1} C \foral z \in \bD^d,
\end{align}
where $\Delta(z) := z_1 P_1 + \dots + z_d P_d$ and $P_j$ is the orthogonal projection onto $\cH_j$. This framework was later explored over $\bB_d$ by Ball, Trent, and Vinnikov in \cite{Ball-Trent-Vinnikov} where they obtained a realization formula for functions in the unit ball of the multiplier algebra of the Drury--Arveson space (see Section \ref{sec:zeros.hol.func.unit.ball}), which, as in the polydisk case, forms a subclass of $\cS(\bB_d)$. They showed that $f \in \cM(\bB_d)_1$ if and only if there exist an auxiliary Hilbert space $\cH$ and a unitary colligation
$$ V =
\begin{bmatrix}
        A & B \\
        C_1 & D_1 \\
        \vdots & \vdots \\
        C_d & D_d
    \end{bmatrix} :
    \begin{bmatrix}
        \bC \\ \cH
    \end{bmatrix}
    \to
    \begin{bmatrix}
        \bC \\ \cH \otimes \bC^d
    \end{bmatrix}
$$
such that
\begin{equation}\label{eqn:intro.realization.form.unit.ball}
    f(z) = A + B \Bigl(I - \sum_{j=1}^d z_j D_j \Bigr)^{-1} 
    \Bigl(\sum_{j=1}^d z_j C_j \Bigr) \foral z=(z_1, \dots, z_d) \in \bB_d.
\end{equation}

\subsubsection*{\texorpdfstring{\textbf{For non-commutative functions}}{For non-commutative functions}}

In a quest to generalize these formulae to other domains, Ambrozie and Timotin \cite{Ambrozie-Timotin-vonNeumann-inequality} considered the Schur--Agler class over
\begin{equation*}
    \bD_Q := \{ z \in \bC^d : \|Q(z)\| < 1 \},
\end{equation*}
where $Q$ is an $s \times r$ matrix of linear polynomials in $d$ variables. We call these domains  \emph{matrix unit balls}. Note that $\bD_Q = \bD^d$ for $Q(z) = \diag(z_1,\dots,z_d)$, and $\bD_Q = \bB_d$ for $Q(z) = [z_1 \, \dots \, z_d]$. These domains have been explored in greater detail within the modern framework of \emph{free analysis}, where the ambient space $\mathbb{C}^d$ is replaced by the non-commutative (NC) universe $\bM^d$ (defined below), and holomorphic functions are replaced by NC functions that satisfy a mild local boundedness condition. Taylor worked out the functional calculus of non-commuting tuples in the 1970s and introduced certain axioms to define non-commutative functions, as a conceptual generalization of commutative function theory \cite{Taylor-NC-functional-calc, Taylor-NC-functions}. Even though Taylor's work went largely unnoticed at the time, it has now proven to be successful with several works demonstrating its utility in operator theory, free probability, and even systems/control theory \cite{Ball-2006, Helton-Vinnikov, Popescu-JFA-2006, Voiculescu}. Several prominent works such as those of Agler and McCarthy \cite{Agler-mcCarthy-2006, Agler-McCarthy-Young-Book}, Helton, Klep and McCulough \cite{Helton-Klep-McCullough-JMAA, Helton-Klep-McCullough-JFA, Helton-Klep-McCullough-Book-2012}, Kaliuzhnyi-Verbovetskyi and Vinnikov \cite{Vinnikov-Verbovetskyi-book}, Ball, Marx and Vinnikov \cite{Ball-Marx-Vinnikov-JFA-2006, Ball-Marx-Vinnikov-NCTFR}, and Ball and Bolotnikov \cite{Ball-Bolotnikov-NC2007} have shaped this theory into a rich field with several evolving directions.

The fundamentals of NC function theory will be provided in Section \ref{subsec:NC.background}. For the time being, let $\bM^d$ be the graded/disjoint union of $M_{n \times n} \otimes \bC^d$ -- the space of all $d$-tuples of $n \times n$ matrices, and let $Q$ be an $s \times r$ matrix of linear polynomials in $d$ non-commuting variables $Z = (Z_1, \dots, Z_d)$, i.e., $Q(Z) = \sum_{j = 1}^d Q_j Z_j$ for some $Q_j \in M_{s \times r}$, $1 \leq j \leq d$. We then define the NC matrix unit ball
\begin{equation*}
    \bD_Q := \{ X \in \bM^d : \|Q(X)\| < 1 \}.
\end{equation*}
It was shown in \cite{Ball-2006, Ball-Marx-Vinnikov-NCTFR} that the NC analogue of the Schur--Agler class over $\bD_Q$ is the unit ball of $H^\infty(\bD_Q)$ -- the collection of all bounded NC functions $f : \bD_Q \to \bM^1$ (see Section \ref{subsubsec:NC.func}). Moreover, we know from \cite[Remark 2.21 and Corollary 3.2]{Ball-Marx-Vinnikov-NCTFR} that $f \in \cS\cA(\bD_Q) := H^\infty(\bD_Q)_1$ if and only if there is an auxillary Hilbert space $\cH$ and a unitary colligation
\begin{equation*}
    V := \bmat{A & B \\ C & D} : \bmat{\bC \\ \bC^s \otimes \cH} \to \bmat{\bC \\ \bC^r \otimes \cH}
\end{equation*}
such that
\begin{equation}\label{eqn:intro.NC.realization.formula}
    f(X) = A^{(n)} + B^{(n)}[I - (Q(X) \otimes I_\cH) D^{(n)}]^{-1}(Q(X) \otimes I_\cH)C^{(n)}
\end{equation}
for all $X \in \bD_Q \cap (M_{n \times n} \otimes \bC^d)$ and $n \in \bN$, where we use the notation $T^{(n)}= T \otimes I_n$. A different flavor of realization for entire/meromorphic functions in the NC setting have appeared recently in the following works \cite{Augat-Martin-Shamovich-TFR, Martin-Max}.

Over the years many efforts have been made to better understand the Schur--Agler class in commuting and non-commuting variables; see, for example, the following recent works and the references therein \cite{Aleman-Hartz-McCarthy-Richter, Alpay-Bolotnikov-Kaptanouglu, Anderson-Dritschel-Rovnyak, Ball-Bolotnikov-NC2007, Ball-Bolotnikov-Fang-TFR, Ball-Marx-Vinnikov-NCTFR, Barik-Bhattacharjee-Das, Clouatre-Hartz, Clouatre-Kumar, Eschmeier-Putinar, Fang, Greg-IUMJ, Knese-2025, Kojin}. Despite these advances, much about the Schur--Agler classes on general domains remains a mystery. In this paper, we use these realization formulae to provide a new look at the Schur--Agler functions by generalizing Theorem \ref{thm:disk.case} to all the aforementioned cases and exploring their boundary behavior.



\section{Main Results}\label{subsec:main.results}

In what follows, we present our main results, arranged section by section. 
\subsection*{\texorpdfstring{\textbf{Section \ref{sec:commutative case}}}{Section 3}}

In the case $d = 1$, we know that the zero set of a polynomial must be a finite set and is therefore compact. Consequently, since we can dilate any $p \in \bC[z]$ via $z \mapsto p_t(z) := p(tz)$ for some large enough $t$ and rescale it by a factor of $R > 0$ so that $\frac{1}{R}p_t \in \cS(\bD)$, the zero set of $p$ can be captured by eigenvalues of a certain operator using Theorem \ref{thm:disk.case}. The following example shows that this immediately fails for $d > 1$. Let $d = 2$ and $p(z_1,z_2) = (z_1 - \lambda_1)(z_2 - \lambda_2)$ for some $\lambda_1, \lambda_2 \in \bC$. Then, clearly, its zero set is given by $$\{\lambda_1\} \times \bC \; \cup \; \bC \times \{\lambda_2\}.$$ In particular, however much we dilate and rescale $p$ so that it lies in $\cS(\Omega)$ for some bounded domain $\Omega$, we shall not be able to capture all of the zeros of $p$ and, therefore, the zero set of $p$ cannot be contained within a spectrum of some operator in the traditional sense, since spectrums are usually defined as compact sets. We must therefore introduce a new notion of spectrum/eigenvalues that allows for a potentially unbounded set to be a spectrum, or, in this case, a set of eigenvalues to be able to characterize the zero set of $p$ (also see Remark \ref{rem:prop.row.eigenvals} and Example \ref{example:famous.zeros}). This motivates the following definition.

\begin{definition}\label{def:row.eigenvals}
    For any row operator $T = [T_1 \dots T_d] : \cH \otimes \bC^d \to \cH$ on a Hilbert space $\cH$, we say that $\lambda = (\lambda_1, \dots, \lambda_d) \in \bC^d$ is a \emph{row eigenvalue} of $T$ if there exists a non-zero vector $v = [v_1 \dots v_d]^t \in \cH \otimes \bC^d$ such that
    \begin{equation}\label{eqn:temp.row}
        Tv = \lambda v := \sum_{j = 1}^d \lambda_j v_j.
    \end{equation}
    In this case, we say that $v$ is a \emph{row eigenvector} for $T$, and write $\sigma_p^{\row}(T)$ for the collection of all the row eigenvalues of $T$. 
\end{definition}

If $d = 1$, then the row eigenvalues coincide with the eigenvalues of $T$. If $d > 1$, then $\sigma^{\row}_p(T)$ contains the \emph{joint eigenvalues} of $T$, and exhibits additional structural properties that stay in line with the discussion above the definition (see Lemma \ref{lem:eigenvalues-roweigenvalues}). Our first main result characterizes the zeros of functions in $\cM(\bB_d)_1$. 

\begin{mainthm}\label{mainthm:unit.ball}
    Let $f \in \cM(\bB_d)_1$ be a non-constant function admitting a unitary realization formula as in \eqref{eqn:intro.realization.form.unit.ball} with the colligation $V = \bsmallmat{A & B \\ C & D}$, and let $ D^* : \cH \otimes \bC^d \to \cH$ be the associated row operator for $f$. Then,
    $$\cZ_{\bB_d}(f)=\sigma_p^{\operatorname{row}}(D^*) \cap \bB_d.$$
    
\end{mainthm}

For the unit polydisk, we must employ a different notion of eigenvalues.

\begin{definition}\label{def:diag.eig.values}
    Let $\cH_1, \dots, \cH_d$ be Hilbert spaces and write $\cH = \oplus_{j = 1}^d \cH_j$. We say that $\lambda = (\lambda_1, \dots, \lambda_d) \in \bC^d$ is a \emph{diagonal eigenvalue} for some $T \in \cB(\cH)$ if there exists a non-zero vector $v \in \cH$ (called a \emph{diagoanal eigenvector}) such that
    $$T v = \Delta(\lambda) v.$$
\end{definition}
    
We write $\sigma_p^{\diag}(T)$ for the set of all the diagonal eigenvalues of $T$. Observe that this definition reduces to the standard notion of eigenvalues of the operator $T$ if $d = 1$. This brings us to our next main result.

\begin{mainthm}\label{thm:main:polydisk}
    Let $f \in \cS\cA(\bD^d)$ be a non-constant function admitting a unitary realization formula as in \eqref{eqn:intro.realization.form.polydisk} with a colligation $V = \bsmallmat{A & B \\ C & D}$, and let $D^* \in \cB(\cH)$ be the associated operator on $\cH := \oplus_{j = 1}^d \cH_j$ for $f$. Then, 
    $$\cZ_{\bD^d}(f)=\sigma_p^{\operatorname{diag}}(D^*)\cap\bD^d.$$

\end{mainthm}

The proofs of Theorems \ref{mainthm:unit.ball} and \ref{thm:main:polydisk} also show that the \emph{row/diagonal eigenspace}, i.e., the space of all row/diagonal eigenvectors corresponding to some $\lambda \in \cZ(f)$ is one dimensional and the spanning eigenvector is identified through the realization formula as well (see Remarks \ref{rem:row.eigen.sp.1.dim} and \ref{rem:diag.eigen.sp.1.dim}). We close this section with basic properties of diagonal eigenvalues in Lemma \ref{lemma:diag.eigenval.properties}, and some examples involving \emph{rational inner functions} in Examples \ref{example:famous.zeros}--\ref{example:weird.realization}.

\subsection*{\texorpdfstring{\textbf{Section \ref{sec:zeros.hol.func.unit.Free}}}{Section 4}}

For NC functions, we first need to define what is meant by a zero. A couple different notions of zeros exist in this context, but we shall adopt the notion of \emph{determinantal zeros} that appears in \cite{Helton-Klep-Volcic-free-loci} in the context of polynomial factorization, and in \cite{Jury-Martin-Shamovic-Blaschke} in the context of \emph{Blaschke--Singular--Outer factorization} for the NC analogue of the Hardy space. We therefore define the \emph{(determinantal) zero locus} of any $f \in \cS\cA(\bD_Q)$ as
\begin{equation*}
    \cZ_{\bD_Q}(f) := \{X \in \bD_Q : \det f(X) = 0 \}.
\end{equation*}
Now, since these zeros comprise of $d$-tuples of matrices, we need to be quite liberal with our next definition in calling this object the set of \emph{NC $Q$-eigenvalues}.

\begin{definition}\label{intro.def:NC.Q-eigenvalues}
    Let $T \in \cB(\bC^r \otimes \cH, \bC^s \otimes \cH)$ for some Hilbert space $\cH$, and let $Q$ be an $s \times r$ matrix of linear polynomials in $d$ non-commuting variables. We say that $\Lambda \in M_{n \times n} \otimes \bC^d$ is an NC $Q$-eigenvalue at level $n$ if there exists a non-zero vector $\vec{v} \in \bC^r \otimes \cH \otimes \bC^n$ such that
    \begin{equation*}
        T^{(n)}\vec{v} = (Q(\Lambda) \otimes I_\cH) \vec{v}.
    \end{equation*}
    We write $\sigma^Q_p(T^{(n)})$ for the set of all NC $Q$-eigenvalues of $T$ at level $n$, and
    \begin{equation*}
        \sigma^Q_p(T) := \bigsqcup_{n \in \bN} \sigma^Q_p(T^{(n)}).
    \end{equation*}
\end{definition}

It is clear that $\cZ_{\bD_Q}(f)$ and $\sigma_p^Q(T)$ are NC sets, i.e., closed under direct sums. This leads to our third main result and the main theorem of this section.

\begin{mainthm}\label{mainthm:zeros.NC.Schur--Agler.class.realization}
    Let $\bD_Q \subset \bM^d$ be an NC matrix unit ball, let $f \in \cS\cA(\bD_Q)$ be a non-constant NC function admitting a unitary realization formula as in \eqref{eqn:intro.NC.realization.formula} with a colligation $V = \bsmallmat{A & B \\ C & D}$, and let $D^*$ be the associated operator for $f$. Then,
    \begin{equation*}
        \cZ_{\bD_Q}(f) = \sigma^Q_p(D^*) \cap \bD_Q.
    \end{equation*}
\end{mainthm}

The proof of this result provides a proof of an analogous result for the commutative matrix unit balls. We record this in Theorem \ref{thm:zeros.commutative.D_Q.SA.class} and close this section.

\subsection*{\texorpdfstring{\textbf{Section \ref{sec:boundary.values.approx.pt.spec}}}{Section 5}}

In the final section, we push our techniques further and capture certain behavior of functions in $\cS\cA(\bD_Q)$ for an NC matrix unit ball $\bD_Q$ with the NC $Q$-spectral data. First, note that if $r \in (0,1)$ and $\Lambda \in \bD_Q$ then $r \Lambda \in \bD_Q$ as well. Moreover, the topological boundary of $\bD_Q$ is given by
\begin{equation*}
    \partial \bD_Q := \{ X \in \bM^d : \|Q(X)\| = 1 \}.
\end{equation*}
We then say that $f \in \cS\cA(\bD_Q)$ has a boundary value at some $\Lambda \in \partial \bD_Q$ if the following limit exists:
\begin{equation*}
    f^\dagger(\Lambda) := \lim_{r \to 1} f(r \Lambda).
\end{equation*}
The \emph{boundary zero locus} of $f \in \cS\cA(\bD_Q)$ is then defined as
$$\cZ_{\partial \bD_Q}(f) := \{\Lambda \in \partial \bD_Q : \det f^\dagger(\Lambda) = 0\}.$$
To capture the boundary zeros, it turns out that we need to consider the \emph{NC $Q$-approximate spectrum} which consists of all $\Lambda \in M_{n \times n} \otimes \bC^n$ such that there exists a sequence $\{\vec{v}_k\}_{k \in \bN}$ of unit vectors for which
\begin{equation*}
    \lim_{k \to \infty} \|T^{(n)} \vec{v}_k - (Q(\Lambda) \otimes I_\cH) \vec{v}_k\| = 0.
\end{equation*}
We write $\sigma_{ap}^Q(T)$ for the NC set of the NC $Q$-approximate eigenvalues of $T$ and arrive at the first main result of this section.

\begin{mainthm}\label{mainthm:boundary.zeros.NC}
    Let $\bD_Q \subset \bM^d$ be an NC matrix unit ball, let $f \in \cS\cA(\bD_Q)$ be a non-constant NC function admitting a unitary realization formula as in \eqref{eqn:intro.NC.realization.formula} with a colligation $V = \bsmallmat{A & B \\ C & D}$, and let $D^*$ be the associated operator for $f$. Then,
    \begin{equation*}
        \cZ_{\partial \bD_Q}(f) \subseteq \sigma^Q_{ap}(D^*) \cap \partial \bD_Q.
    \end{equation*}
\end{mainthm}

It is natural to ask if the above inclusion is always an equality, however, we note that $\sigma_{ap}^Q(D^*)$ captures some additional -- rather complicated data. To this end, we introduce the \emph{isometric} and \emph{coisometric portions} of the boundary of $\bD_Q$:
\begin{align*}
    \partial_{\iso}\bD_Q &:= \{ \Lambda \in \bM^d : Q(\Lambda)^* Q(\Lambda) = I_\cU \};\\
    \partial_{\coiso}\bD_Q &:= \{ \Lambda \in \bM^d : Q(\Lambda) Q(\Lambda)^* = I_\cV \}.
\end{align*}
See Definition \ref{def:isometric.boundary.pt.NC} and Remark \ref{rem:Shilov.vs.essential.boundary} for basic properties of these objects, and their relation to the Shilov boundary of $\bD_Q$. For any $f \in \cS\cA(\bD_Q)$, we also define $BP(f,1)$ as the collection of all boundary points 
$\Lambda \in \partial \bD_Q$ such that $f^\dagger(\Lambda)$ exists but is not an isometry, and similarly $BP^*(f,1)$ for $\Lambda \in \partial \bD_Q$ such that $f^\dagger(\Lambda)$ exists but is not a coisometry.

\begin{mainthm}\label{mainthm:NC.boundary.val.approx.pt.spec}
    Let $\bD_Q \subset \bM^d$ be an NC matrix unit ball, let $f \in \cS\cA(\bD_Q)$ be a non-constant NC function admitting a unitary realization formula as in \eqref{eqn:intro.NC.realization.formula} with a colligation $V = \bsmallmat{A & B \\ C & D}$, and let $D^*$ be the associated operator for $f$. Then,
    \begin{equation*}
        \big(BP(f,1) \cap \partial_{\iso}(\bD_Q)\big) \; \cup \; \big(BP^*(f,1) \cap \partial_{\coiso}\bD_Q\big)  \subseteq \sigma^Q_{ap}(D^*).
    \end{equation*}
\end{mainthm}

Analogous versions of the last two theorems in the commutative case are found and recorded in Section \ref{subsec:boundary.val.comm.func}. We close our discussion by considering extensions of the examples from the commutative case (see Examples \ref{example:general.boundary.val} and \ref{example:weird.boundary.val}). In particular, we note that if $f \in \cS\cA(\bD^d)$ is given by a finite dimensional unitary realization, then
\begin{equation*}
    \sigma^{\diag}_{ap}(D^*) = \sigma^{\diag}_p(D^*) = \{\lambda \in \bC^d : \det(D^* - \Delta(\lambda)) = 0\}.
\end{equation*}

\section{The Commutative Case}\label{sec:commutative case}

It is worth highlighting the commutative case before further generalizations for a few reasons. Firstly, it would help readers unfamiliar with NC function theory and the tools therein to build some intuition about the abstract proofs in the next section. Secondly, we intend to highlight a couple of important distinctions between the Euclidean unit ball $\bB_d$ and the unit polydisk $\bD^d$. Observe that the realization formula \eqref{eqn:intro.realization.form.unit.ball} for functions in $\cM(\bB_d)_1$ must always arise from an infinite dimensional auxillary Hilbert space $\cH$, whereas for functions in $\cS\cA(\bD^d)$ one can construct finite dimensional realizations through unitary matrices. On the other hand, row eigenvalues are more tangible -- in a certain sense, and exhibit general properties that the diagonal eigenvalues do not seem to on a first glance.

\subsection{The Euclidean unit ball}\label{sec:zeros.hol.func.unit.ball}
The \emph{Drury--Arveson space} on $\bB_d$ is defined as
\begin{equation*}
    \cH^2_d := \left\{ f \sim \sum_{\alpha \in \bZ_+^d} c_\alpha z^\alpha : \|f\|_{\cH^2_d} := \sum_{\alpha \in \bZ_+^d} |c_\alpha|^2 \frac{\alpha!}{|\alpha|!} < \infty \right\}.
\end{equation*}
It is not difficult to establish that $\cH^2_d$ is a Hilbert space of holomorphic functions on $\bB_d$ with inner-product derived from the $\|.\|_{\cH^2_d}$ norm. Furthermore, it is a reproducing kernel Hilbert space (RKHS) with the kernel
\begin{equation*}
    K_{\cH^2_d}(z,w) := \frac{1}{1 - \ip{z,w}_{\bC^d}} \foral z,w \in \bB_d.
\end{equation*}
As is the case with many RKHSs, $\cH^2_d$ comes with a rich \emph{multiplier algebra}
\begin{equation*}
    \cM(\bB_d) := \{ \varphi : \bB_d \to \bC : f \in \cH^2_d \implies \varphi f \in \cH^2_d \}.
\end{equation*}
It is a standard fact that $\cM(\bB_d) \subsetneq H^\infty(\bB_d)$ for $d > 1$, and that it turns into a Banach algebra under the norm
\begin{equation*}
    \|\varphi\|_\cM := \|M_\varphi\|_{\cB(\cH^2_d)} \foral \varphi \in \cM(\bB_d),
\end{equation*}
where $M_\varphi \in \cB(\cH^2_d)$ is given by $M_\varphi : f \mapsto \varphi f$. In recent years, there has been a significant development in the study of the Drury--Arveson space and in several different fields. The interested reader is directed to the thorough and well-written survey by Shalit \cite{Shalit-DA-sp-survey}.

Recall from the introduction that $f \in \cM(\bB_d)_1$ -- the unit ball of $\cM(\bB_d)$ if and only if there exists an auxillary Hilbert space $\cH$ and a unitary colligation
\begin{equation*}
   V = \bmat{A & B \\ C_1 & D_1 \\ \vdots & \vdots \\ C_d & D_d} : \bmat{\bC \\ \cH} \to \bmat{\bC \\ \cH \otimes \bC^d}
\end{equation*}
such that
\begin{equation} \label{eqn:realization:ball}
    f(z) = A + B (I - z D)^{-1} (z C ) \foral z \in \bB_d,
\end{equation}
where we use the following notation for convenience:
\begin{equation*}
    z C := \sum_{j = 1}^d z_j C_j \qquad \AND \qquad z D := \sum_{j = 1}^d z_j D_j.
\end{equation*}

\subsubsection*{\texorpdfstring{\textbf{Proof of Theorem \ref{mainthm:unit.ball}}}{Proof of Theorem A}}

Let $f$ be as in \eqref{eqn:realization:ball} with a unitary colligation $V = \bsmallmat{A & B \\ C & D}$ and let $\lambda \in \cZ_{\bB_d}(f)$ be given. We introduce the vectors
    \begin{equation*}
        L(\lambda) := (I - \lambda D)^{-1} (\lambda C) \in \cH\qand v^\lambda := C + DL(\lambda) \in \cH \otimes \bC^d,
    \end{equation*}
    and note that
    \begin{align}
        \lambda v^\lambda &= \lambda C + \lambda D L(\lambda)) \nonumber \\
        &= \lambda C + \lambda D [ (I - \lambda D)^{-1} (\lambda C)] \nonumber \\
        &= [I + \lambda D (I - \lambda D)^{-1}] (\lambda C) \nonumber \\
        &= L(\lambda). \label{eqn:calculation.L(lambda)} 
    \end{align}
    It follows that
    \begin{equation}\label{eqn:zeros.ball.proof.1}
        \bmat{A & B \\ C & D} \bmat{1 \\ L(\lambda)} = \bmat{f(\lambda) \\ C + D L(\lambda)} = \bmat{0 \\ v^\lambda}.
    \end{equation}
    Applying $V^*$ to the left on both sides of \eqref{eqn:zeros.ball.proof.1} gives us
    \begin{align*}
        \bmat{1 \\ L(\lambda)}= \bmat{A^* & C^* \\ B^* & D^*} \bmat{0 \\ v^\lambda} = \bmat{C^* v^\lambda \\ D^* v^\lambda}.
    \end{align*}
    Comparing the two blocks above and using \eqref{eqn:calculation.L(lambda)} provides these relations: 
    $$C^*v^\lambda = 1 \qquad \AND \qquad D^* v^\lambda = L(\lambda) = \lambda v^\lambda.$$
    Since $C^* v^\lambda = 1$, we get $v^\lambda \neq 0$ and then the second relation implies $\lambda \in \sigma^{\row}_p(D^*)$. 

    Conversely, suppose $\lambda = (\lambda_1, \dots, \lambda_d)\in\sigma_p^{\operatorname{row}}(D^*)\cap\mathbb{B}_d$ and let $v^\lambda \in \cH \otimes \bC^d$ be a non-zero vector such that $D^*v^\lambda=\lambda v^\lambda$. Then, we observe that
    \begin{equation*}
        \bmat{A & B \\ C & D} \bmat{0 \\ v^\lambda} = \bmat{C^* v^\lambda \\ D^* v^\lambda} = \bmat{C^* v^\lambda \\ \lambda v^\lambda}.
    \end{equation*}
    Applying $V$ to the left on both sides above gives us 
    \begin{equation*}
        \bmat{0 \\ v^\lambda} = \bmat{A & B \\ C & D} \bmat{C^* v^\lambda \\ \lambda v^\lambda} = \bmat{(C^* v^\lambda) A + B \lambda v^\lambda \\ (C^* v^\lambda) C  + D \lambda v^\lambda}.
    \end{equation*}
    Comparing both the blocks above gives us the following two relations:
    \begin{equation}\label{eqn:calculation.ball.proof.1}
        (C^* v^\lambda) A + B \lambda v^\lambda=0 \qquad \AND \qquad (C^* v^\lambda) C + D \lambda v^\lambda = v^\lambda.
    \end{equation}
    Note that if $C^*v^\lambda=0$, then the second relation above implies $D \lambda v^\lambda = v^\lambda$, which is not possible since $D$ is a contraction and
    \begin{equation*}
        \|\lambda v^\lambda\| \leq \|\lambda\|_{\operatorname{row}} \|v^\lambda\|_{\operatorname{col}} < \|v^\lambda\|_{\operatorname{col}}.
        \end{equation*}
    We therefore assume WLOG that $C^*v^\lambda=1$, so that \eqref{eqn:calculation.ball.proof.1} becomes 
    \begin{equation*}
        A + B \lambda v^\lambda = 0 \qquad \AND \qquad C + D (\lambda v^\lambda)=v^\lambda.
    \end{equation*}
    It easily follows from a calculation similar to \eqref{eqn:calculation.L(lambda)} and the second relation above that $\lambda v^\lambda = L(\lambda)$. Plugging this into the first relation above, we get 
    \begin{equation*}
        0=A+ B \lambda v^\lambda = A + B L(\lambda) = f(\lambda).
    \end{equation*}
    This implies that $\lambda \in \cZ_{\bB_d}(f)$, as required. \hfill \qedsymbol

\begin{remark}\label{rem:row.eigen.sp.1.dim}
    For each $\lambda \in \cZ_{\bB_d}(f)$, we have inadvertently shown that the row eigenspace corresponding to $\lambda$ is spanned by $v^\lambda := C + D [(I - \lambda D)^{-1}(\lambda C)]$.
\end{remark}

\subsubsection*{\texorpdfstring{\textbf{Properties of row eigenvalues}}{Properties of row eigenvalues}}

Let $T = [T_1 \; \dots \; T_d] : \cH \otimes \bC^d \to \cH$ be a row operator over some Hilbert space $\cH$. Recall that the set $\sigma_p(T)$ of \emph{joint eigenvalues} of $T$ consists of all $\lambda = (\lambda_1, \dots, \lambda_d) \in \bC^d$ such that there exists $0 \neq v \in \cH$ for which $T_j v = \lambda_j v$, for all $j$. Recall also that $T$ is said to be \emph{jointly similar} to another row operator $S = [S_1 \; \dots \; S_d] : \cK \otimes \bC^d \to \cK$ if there exists an invertible operator $P \in \cB(\cH,\cK)$ such that $T_j = P^{-1} S_j P$, for all $j$. We use the notation $T \sim_P S$ to denote joint similarity and write $T = P^{-1} S P$.

\begin{lemma}\label{lem:eigenvalues-roweigenvalues}
    Let $T = [T_1 \dots T_d] : \cH \otimes \bC^d \to \cH$ and $S = [S_1 \dots S_d] : \cK \otimes \bC^d \to \cK$ be row operators on some Hilbert spaces $\cH, \cK$. Then, the following holds:
    \begin{enumerate}
        \item $\sigma_p(T) \subset \sigma^{\row}_p(T)$.
        \item If $\lambda_k \in \sigma_p(T_k)$ for some $1 \leq k \leq d$, then
        \begin{equation*}
            \bC^{k-1}\times\{\lambda\}\times\bC^{d-k} \subset \sigma^{\row}_{p}(T)
        \end{equation*}
        \item More generally, if $(\lambda_{j_1}, \dots, \lambda_{j_k}) \in \sigma^{\row}_p(T_{j_1}, \dots, T_{j_k})$ for some $1 \leq j_l < j_{l + 1} \leq d$ and $1 \leq l < k < d$, then
        \begin{equation*}
            \bC^{j_1 - 1} \times \{\lambda_{j_1}\} \times \bC^{j_2 - j_1 - 1} \times \dots \times \bC^{j_k - j_{k - 1} - 1} \times \{\lambda_{j_k}\} \times \bC^{d - j_k} \subset \sigma^{\row}_p(T).
        \end{equation*}
        \item If $T \sim_P S$, then $\sigma^{\row}_p(T) = \sigma^{\row}_p(S)$.
    \end{enumerate}
\end{lemma}

\begin{proof}
    Let $T$ and $S$ be as in the hypothesis.
    \begin{enumerate}
        \item Suppose $\lambda = (\lambda_1, \ldots, \lambda_d) \in \sigma_p(T)$ and let $0 \neq w \in \cH$ be such that $T_j w = \lambda_j w$ for all $1 \leq j \leq d$. Set $v=[w \; \dots \; w]^t$ and observe that
        $$Tv= \sum_{j=1}^dT_jw=\sum_{j=1}^d\lambda_jw = \lambda v.$$
        \item Choose $0 \neq v_k \in \cH$ such that $T_k v_k = \lambda_k v_k$. Fix arbitrary scalars $\lambda_j \in \mathbb{C}$ for $1 \leq j \neq k \leq d$, and let $v = [0, \dots, 0, v_k, 0, \dots, 0],$ with $v_k$ in the $k^{\text{th}}$ position. Thus, we get $T v = \lambda_k v_k = \lambda v$.
        \item The proof of this is the same as $(2)$, except we take $0 \neq w = (w_{j_1},\dots,w_{j_k})$ such that $[T_{j_1} \dots T_{j_d}] w = [\lambda_{j_1} \dots \lambda_{j_d}] w$ and define
        \begin{equation*}
            v = [0, \dots, 0, w_{j_1}, 0, \dots, 0, w_{j_d}, 0, \dots, 0]^t \in \cH \otimes \bC^d
        \end{equation*}
        with each $w_{j_l}$ at the $j_l^{\text{th}}$ position. Then, for arbitrarily chosen $\lambda_{j'} \in \bC$ with $1 \leq j' \neq j_l \leq d$ for any $1 \leq l \leq k$, we see that $Tv = \lambda v$.
        \item Let $\lambda \in \bC^d$ and $0 \neq [v_1 \; \dots \; v_j] \in \cH \otimes \bC^d$. Set $w_j = P v_j$ and note that
        \begin{align*}
            (T - \lambda I) v = 0 \iff P^{-1} (S - \lambda I)Pv = 0 \iff (S - \lambda I)w = 0.
        \end{align*}
        The invertibility of $P$ then easily implies the result.\qedhere
    \end{enumerate}
\end{proof}

\begin{remark}\label{rem:prop.row.eigenvals}
    Lemma \ref{lem:eigenvalues-roweigenvalues} $(2)$ shows that $\sigma^{\row}_p(T)$ can be an unbounded set and cannot, in general, lie inside any traditionally defined joint spectrum. In the context of Theorem \ref{mainthm:unit.ball}, if $d = 2$ and $\lambda_1 \in \sigma_p(D_1^*)$ for some $D^* = [D_1^* \; D_2^*]$ associated to a unitary realization of some $f \in \cM(\bB_2)_1$, then $(2)$ above suggests that $f$ `contains' a factor of $(z_1 - \lambda_1)$ or some modification of it to allow the remaining function to lie in $\cM(\bB_2)_1$. This, combined with the generalization in $(3)$ warrants future investigation about the structure of row eigenvalues of the associated operator $D^*$ for any given $f \in \cM(\bB_d)_1$ and its relation to factorizations of $f$.
\end{remark}
    
\subsection{The unit polydisk}\label{sec:zeros.hol.func.unit.[polydisk]}

The \emph{Hardy space} on $\bD^d$ is defined as
\begin{equation*}
    H^2(\bD^d) := \left\{ f \sim \sum_{\alpha \in \bZ_+^d} c_\alpha z^\alpha : \|f\|_{H^2(\bD^d)} = \sum_{\alpha \in \bZ_+^d} |c_\alpha|^2 < \infty \right\}.
\end{equation*}
Just like the Drury--Arveson space, it is straight-forward to check that $H^2(\bD^d)$ is a Hilbert space of holomorphic functions on $\bD^d$ with inner-product derived from the $\|.\|_{H^2(\bD^d)}$ norm. Furthermore, it is a RKHS with the kernel
\begin{equation*}
    K_{H^2(\bD^d)}(z,w) := \prod_{j = 1}^d \frac{1}{1 - \overline{w_j} z_j} \foral z,w \in \bD^d.
\end{equation*}
For more information on Hardy spaces over $\bD^d$ see Rudin's book \cite{Rudin-Book-Polydisk}.

It is straight-forward to check that the multiplier algebra of $H^2(\bD^d)$ is simply $H^\infty(\bD^d)$ -- the collection of all bounded holomorphic maps on $\bD^d$. As noted in the introduction, we need to work with the Schur--Agler class $\cS\cA(\bD^d) \subset \cS(\bD^d)$ (strict containment only if $d > 2$) in order to obtain a realization formula. We saw that $f \in \mathcal{SA}(\bD^d)$ if and only if there exist auxillary Hilbert spaces $\cH_1, \dots, \cH_d$ and a unitary colligation
\begin{equation*}
    V = \bmat{A & B \\ C & D} : \bmat{\bC \\ \oplus_{j = 1}^d \cH_j} \to \bmat{\bC \\ \oplus_{j = 1}^d \cH_j}
\end{equation*}
such that
\begin{equation}\label{eqn:realization.form.polydisk}
    f(z) = A + B \Delta(z) (I - D \Delta(z))^{-1} C \foral z \in \bD^d,
\end{equation}
where $\Delta(z) := z_1 P_1 + \dots + z_d P_d$ and $P_j$ is the orthogonal projection onto $\cH_j$.

\subsubsection*{\texorpdfstring{\textbf{Proof of Theorem \ref{thm:main:polydisk}}}{Proof of Theorem B}}
    
Let $f$ be as in \eqref{eqn:realization.form.polydisk} with a unitary colligation $V = \bsmallmat{A & B \\ C & D}$ and let $\lambda \in \mathcal{Z}_{\mathbb{D}^d}(f)$ be given. We introduce the vector
$$v^\lambda :=\left(I-D\Delta(\lambda)\right)^{-1}C \in \cH := \oplus_{j = 1}^d \cH_j,$$
and note as in the proof of Theorem \ref{mainthm:unit.ball} that
\begin{equation}\label{eqn:polydisk.zero.proof.1}
    \bmat{A & B \\ C  & D }\bmat{1 \\ \Delta(\lambda)v^\lambda}=\bmat{f(\lambda) \\ C + D \Delta(\lambda)v^\lambda}=\bmat{0 \\ v^\lambda}.
\end{equation}
The last block in the second equality above holds since
\begin{alignat*}{2}
    && (I - D\Delta(\lambda))v^\lambda &= C, \\
    \implies && v^\lambda &= C + \Delta(\lambda) v^\lambda.
\end{alignat*}
Applying $V^*$ to both sides of \eqref{eqn:polydisk.zero.proof.1}, we get
$$\bmat{ 1 \\ \Delta(\lambda)v^\lambda}=\bmat{A^* & C^* \\ B^* & D^*}\bmat{ 0 \\ v^\lambda}=\bmat{C^*v^\lambda \\ D^*v^\lambda}.$$
This gives the following two relations: 
$$C^*v^\lambda=1 \qquad \AND \qquad D^*v^\lambda=\Delta(\lambda)v^\lambda.$$
Since $C^* v^\lambda = 1$, we get $v^\lambda \neq 0$ and then the second relation implies $\lambda \in \sigma^{\diag}_p(D^*)$.

Conversely, suppose $\lambda \in \sigma_p^{\operatorname{diag}}(D^*)\cap\bD^d$ and let $v^\lambda\in \cH$ be a non-zero vector such that $D^*v^\lambda=\Delta(\lambda)v^\lambda$.  Observe that 
\begin{align*}
        \begin{bmatrix}
            A^* & C^* \\
            B^* & D^*
        \end{bmatrix}  \begin{bmatrix}
            0 \\
            v^\lambda
        \end{bmatrix}=\begin{bmatrix}
            C^*v^\lambda \\
            D^*v^\lambda
        \end{bmatrix} = \begin{bmatrix}
            C^*v^\lambda \\
            \Delta(\lambda)v^\lambda 
        \end{bmatrix}.
\end{align*}
Applying $V$ to both sides of the above equation, we get 
\begin{equation}\label{eqn:polydisk.proof.1}
    \bmat{ 0 \\ v^\lambda}=\bmat{ A & B \\ C & D} \bmat{C^*v^\lambda \\ \Delta(\lambda)v^\lambda}=\bmat{A(C^*v^\lambda)+B\Delta(\lambda)v^\lambda\\
    C(C^*v^\lambda)+D\Delta(\lambda)v^\lambda}
\end{equation}
As before, if $C^* v^\lambda = 0$ then
\begin{equation*}
    D \Delta(\lambda ) v^\lambda = v^\lambda,
\end{equation*}
which cannot happen since $D$ is a contraction and
\begin{equation*}
    \|\Delta(\lambda)v^\lambda\| \leq \|\lambda\|_\infty\|v^\lambda\| < \|v^\lambda\|.
\end{equation*}
Therefore, we can assume without loss of generality that $C^* v^\lambda = 1$ and rewrite \eqref{eqn:polydisk.proof.1} to obtain the following two relations:
$$A+B\Delta(\lambda)v^\lambda=0 \quad \text { and } \quad C+D\Delta(\lambda)v^\lambda=v^\lambda.$$
Solving the second equation for $v^\lambda$ and substituting it into the first shows that
$$f(\lambda) = A+B\Delta(\lambda)(I-D\Delta(\lambda))^{-1}C = 0.$$
This shows that $\lambda \in \mathcal{Z}_{\mathbb{D}^d}(f)$, as required. \hfill \qedsymbol

\begin{remark}\label{rem:diag.eigen.sp.1.dim}
    For each $\lambda \in \cZ_{\bD^d}(f)$, it follows from the proof above that the diagonal eigenspace corresponding to $\lambda$ is spanned by $v^\lambda := (I - D \Delta(\lambda))^{-1} C$.
\end{remark}

\subsubsection*{\texorpdfstring{\textbf{Properties of diagonal eigenvalues}}{Properties of diagonal eigenvalues}}

Let $\cH_1, \dots, \cH_d$ be Hilbert spaces and write $\cH = \oplus_{j = 1}^d \cH_j$. Even though we cannot say much about the diagonal eigenvalues of a general operator $T \in \cB(\cH)$ as we did in Lemma \ref{lem:eigenvalues-roweigenvalues} for the row eigenvalues, we can make a similar observation if $T$ admits a block upper triangular structure as follows:
\begin{equation}\label{eqn:block.triangle.struct.T}
    T = \bmat{T_{11} & T_{12} & \dots & T_{1d} \\
              0 & T_{22} & \dots & T_{2d} \\
              \vdots & \vdots & \ddots & \vdots \\
              0 & 0 & \dots & T_{dd}} \in \cB\Big(\bigoplus_{j = 1}^d \cH_j\Big)
\end{equation}
This scenario appears in situations where one wishes to use the realization of a given function $f \in \cS\cA(\bD^d)$ and obtain a realization for $f_1, f_2 \in \cS\cA(\bD^d)$ such that $f = f_1 f_2$. See \cite{Sz.-Nagy-Foias-Book} as well as the recent works \cite{Bhowmik-Kumar, Ramlal-Sarkar} for more information on this topic.

\begin{lemma}\label{lemma:diag.eigenval.properties}
    Suppose $T \in \cB(\cH)$ admits a block upper triangular structure as in \eqref{eqn:block.triangle.struct.T} and let $1 \leq k \leq d$ be arbitrary. Define $T^{(k)}$ to be the $k \times k$ top-left sub-block of $T$ consisting of $T_{11}$ through $T_{kk}$ on its main diagonal, and similarly define $T_{(k)}$ to be the bottom-right $k \times k$ sub-block of $T$ consisting of $T_{d-k+1,d-k+1}$ through $T_{dd}$ on its main diagonal. Then,
    \begin{align*}
        \sigma^{\diag}_p(T^{(k)}) \times \bC^{d - k} &\subset \sigma^{\diag}_p(T); \\
        \bC^{d - k} \times \sigma^{\diag}_p(T_{(k)}) &\subset \sigma^{(diag)}_p(T).
    \end{align*}
\end{lemma}

\begin{proof}
    Suppose $\lambda^{(k)} \in \sigma^{\diag}_p(T^{(k)})$ with a non-zero vector $v^{(k)} \in \oplus_{j = 1}^{k} \cH_j$ such that
    \begin{equation*}
        T^{(k)} v^{(k)} = \Delta^{(k)}(\lambda^{(k)}) v^{(k)},
    \end{equation*}
    where 
    \begin{equation*}
        \Delta^{(k)}(\lambda^{(k)}) := \lambda^{(k)}_1 P_1 + \dots + \lambda^{(k)}_k P_k.
    \end{equation*}
    Then, write
    \begin{equation*}
        T = \bmat{T^{(k)} & * \\  0 & T_{(d-k)}} \in \cB\big((\oplus_{j = 1}^{k} \cH_j) \oplus (\oplus_{j = k + 1}^{d} \cH_j) \big)
    \end{equation*}
    and pick $\lambda_{j'} \in \bC$ arbitrary for all $k + 1 \leq j' \leq d$. If we define
    \begin{equation*}
        0 \neq v = \bmat{v^{(k)} \\ 0} \in \cH,
    \end{equation*}
    then note that
    \begin{equation*}
        Tv = \bmat{T^{(k)} & * \\ 0 & T_{(d - k)}} \bmat{v^{(k)} \\ 0} = \bmat{T^{(k)} v^{(k)} \\ 0} = \Delta(\lambda) v.
    \end{equation*}
    Thus, $\lambda \in \sigma^{\diag}_p(T)$ as required.
    
    The proof for the bottom-right $k \times k$ sub-block $T_{(k)}$ is similar.
\end{proof}

\subsubsection*{\texorpdfstring{\textbf{Examples: Rational inner functions}}{Examples: Rational inner functions}}

We conclude this section with some examples. Recall from the introduction that a map $\varphi \in \cS(\bD^d)$ is said to be inner if $|\varphi^\dagger| = 1$ Lebesgue a.e. on $\bT^d$. Particularly interesting case of inner functions is when $\varphi$ is also a rational function as they appear in several different areas of math. The recent excellent survey from Knese \cite{Knese-Survey} is a great source of information about rational inner functions.

\begin{example}\label{example:famous.zeros}
    The ``famous example'' is the rational inner function  
    \begin{equation*}
        f(z, w) = \frac{2zw - z - w}{2 - z - w} \in \cS(\bD^2) = \cS\cA(\bD^2).
    \end{equation*}
    Observe that $f$ has a non-removable singularity at the point $(1,1) \in \bT^2$. Using Agler’s approach, we obtain a unitary realization for $f$ as follows:
    $$B = \begin{bmatrix} -\tfrac{1}{\sqrt{2}} & -\tfrac{1}{\sqrt{2}} \end{bmatrix}, \quad C = \begin{bmatrix} \tfrac{1}{\sqrt{2}} \\[6pt] \tfrac{1}{\sqrt{2}} \end{bmatrix}, \quad D = \begin{bmatrix} \tfrac{1}{2} & -\tfrac{1}{2} \\[6pt] -\tfrac{1}{2} & \tfrac{1}{2} \end{bmatrix}.
    $$
    This provides us with a unitary colligation matrix
    $$
    V = 
    \begin{bmatrix}
        0 & B \\[6pt]
        C & D
    \end{bmatrix}
    : \mathbb{C} \oplus \mathbb{C}^2 \,\longrightarrow\, \mathbb{C} \oplus \mathbb{C}^2,$$
    and it is straight-forward to check that this is a realization for $f$. Now, note that
    \begin{align*}
        (\lambda, \mu) \in \sigma^{\diag}_p(D^*) &\iff \operatorname{Ker} [D^* - \Delta(\lambda, \mu)] \neq \{0\} \\
        &\iff \det \bmat{\frac{1}{2} - \lambda & -\frac{1}{2} \\ -\frac{1}{2} & \frac{1}{2} - \mu} = 0 \\
        &\iff 2 \lambda \mu - \lambda - \mu = 0,
    \end{align*}
    which implies that $\sigma^{\diag}_p(D^*)$ is precisely the zero set of the numerator of $f$. In particular, $\sigma^{\diag}_p(D^*)$ captures all zeros of $f$ in $\overline{\bD^2}$ and the singularity at $(1,1)$.
\end{example}

\begin{example}\label{example:general.zeros}
    We can generalize the above example as follows. Let $f \in \cS\cA(\bD^d)$ be a rational inner function. We know from \cite[Theorem 2.9]{Greg-Publ.M} that $f$ has a finite dimensional unitary realization given by a colligation matrix $V = \bsmallmat{A & B \\ C & D} \in \cB(\bC \oplus \bC^N)$, where $N = \sum_{j = 1}^d N_j$. From Theorem \ref{thm:main:polydisk}, we get $\cZ_{\bD^d}(f) = \sigma_p^{\diag}(D^*) \cap \bD^d$ and hence, if we write
    \begin{equation*}
        D = \bmat{D_{11} & \dots & D_{1 d} \\
                  \vdots & \ddots & \vdots \\
                  D_{d1} & \dots & D_{dd}}
    \end{equation*}
    in the $d \times d$ block form with $D_{jk} : \bC^{N_k} \to \bC^{N_j}$, we then note as in Example \ref{example:famous.zeros} that
    \begin{equation*}
    \lambda \in \sigma_p^{\diag}(D^*) \iff \det
    \bmat{D_{11}^* - \lambda_1 I_{N_1} & \dots & D_{d1}^* \\
          \vdots & \ddots & \vdots \\
          D_{1d}^* & \dots & D_{dd}^* - \lambda_d I_{N_d}} = 0.
    \end{equation*}
    Therefore, we obtain the following determinantal representation for zeros of rational inner functions in $\cS\cA(\bD^d)$:
    \begin{equation}\label{eqn:zero.det.repn.rational}
        \cZ_{\bD^d} (f) = \left\{ \lambda \in \bD^d : \det
        \bmat{D_{11}^* - \lambda_1 I_{N_1} & \dots & D_{d1}^* \\
              \vdots & \ddots & \vdots \\
              D_{1d}^* & \dots & D_{dd}^* - \lambda_d I_{N_d}} = 0 \right\}.
    \end{equation}

    It was shown in \cite[Theorem 2.11]{Greg-Publ.M} that the rational inner function
    \begin{equation*}
        f(z_1,z_2,z_3) = \frac{3z_1 z_2 z_3 - z_1 z_2 - z_2 z_3 - z_1 z_3}{3 - z_1 - z_2 - z_3}
    \end{equation*}
    lies in $\cS\cA(\bD^3)$, and that any finite dimensional unitary realization of $f$ as defined in the beginning of this example must have $N \geq 6$. Thus, the polynomial obtained by taking the determinant in the R.H.S. of \eqref{eqn:zero.det.repn.rational} must have degree at least $6$, and so it cannot be the numerator of $f$.

    There is also the question of whether something can be said about the singularity points of a general rational inner function $f$ along $\partial \bD^d$ as we did with the function in Example \ref{example:famous.zeros}. This will be explored in Section \ref{sec:boundary.values.approx.pt.spec}.
\end{example}

In the final example, we showcase how one can use Theorem \ref{thm:main:polydisk} to obtain realization formulae for certain rational inner functions on $\bD^2$. We remark that Theorem \ref{thm:main:polydisk} is not strictly needed for this example but it simplifies the argument by providing appropriate `guesses'.

\begin{example}\label{example:weird.realization}
    Let $\alpha, \beta >0$ be such that $\alpha + \beta = 1$. We explicitly construct a $3 \times 3$ unitary realization for the function
    \begin{equation*}
        f = f_{\alpha,\beta}(z,w) := \frac{zw - \alpha z - \beta w}{1 - \beta z - \alpha w}.
    \end{equation*}
    Note that $f$ is a rational inner function (using \cite[Theorem 5.2.4]{Rudin-Book-Polydisk}) and therefore $f \in \cS(\bD^2)$. It follows from Kummert's theorem \cite{Kummert-1989} that $f$ has a unitary realization with a colligation
    \begin{equation*}
        V = \bmat{0 & b_1 & b_2 \\ c_1 & d_{11} & d_{12} \\ c_2 & d_{21} & d_{22}} \in \cB(\bC \oplus \bC^2).
    \end{equation*}
Let us suppose $B, C, D$ are all real matrices for the time being, where
    \begin{equation*}
        B = \bmat{b_1 & b_2}; \qquad C = \bmat{c_1 \\ c_2}; \qquad D = \bmat{d_{11} & d_{12} \\ d_{21} & d_{22}}.
    \end{equation*}
Note from Example \ref{example:general.zeros} that for $(\lambda, \mu) \in \bD^2$ we must have
    \begin{align*}
        \lambda \mu - \alpha \lambda - \beta \mu = 0 &\iff \det \bmat{d_{11} - \lambda & d_{21} \\ d_{12} & d_{22} - \mu} = 0 \\
        &\iff \lambda \mu - d_{22} \lambda - d_{11} \mu + (d_{11} d_{22} - d_{12} d_{21}) = 0.
    \end{align*}
    We must therefore have $d_{11} = \beta$, $d_{22} = \alpha$, and
    \begin{equation*}
        d_{12} d_{21} = d_{11} d_{22} = \alpha \beta.
    \end{equation*}
    Also, since $V^* V = V V^* = I_3$, we can compute the trace of $V^* V$ and conclude
    \begin{equation*}
        d_{11}^2 + d_{12}^2 + d_{21}^2 + d_{22}^2 = 1 \implies d_{12}^2 + d_{21}^2 = 1 - \alpha^2 - \beta^2 = 2 \alpha \beta.
    \end{equation*}
    We can solve the above two equations for $d_{12}$, $d_{21}$ and note that
    \begin{equation*}
        d_{12} = \pm \sqrt{\alpha \beta} = d_{21}.
    \end{equation*}
    There are only a few cases to check, and it can be verified by using
    \begin{align*}
        b_1^2 = \sqrt{1 - d_{11}^2 - d_{21}^2} &; \quad b_2^2 = \sqrt{1 - d_{12}^2 - d_{22}^2}; \\
        c_1^2 = \sqrt{1 - d_{11}^2 - d_{21}^2} &; \quad c_2^2 = \sqrt{1 - d_{21}^2 - d_{22}^2}
    \end{align*}
    that
    \begin{equation*}
        V = \bmat{0 & \sqrt{\alpha} & \sqrt{\beta} \\ -\sqrt{\alpha} & \beta & -\sqrt{\alpha \beta} \\ -\sqrt{\beta} & -\sqrt{\alpha \beta} & \alpha}
    \end{equation*}
    works as a unitary realization for $f = f_{\alpha, \beta}$.

    Note that this argument also generalizes to the case when $\alpha, \beta \in \bC \setminus \{0\}$ and $|\alpha| + |\beta| = 1$, however some small adjustments need to be made. First, we have
    \begin{equation*}
        f = f_{\alpha, \beta} = \frac{zw - \alpha z - \beta w}{1 - \overline{\beta} z - \overline{\alpha} w},
    \end{equation*}
    and the early calculation with $D^*$ will introduce complex conjugates. We therefore make the following guesses using Theorem \ref{thm:main:polydisk} instead: $d_{11} = \overline{\beta};\; d_{22} = \overline{\alpha}$. The calculation with $d_{12}$ and $d_{21}$ then gives us
    \begin{equation*}
        d_{12} d_{21} = \overline{\alpha \beta}; \qquad |d_{12}|^2 + |d_{21}|^2 = 2 |\alpha \beta|.
    \end{equation*}
    It remains to choose an appropriate square-root of $\overline{\alpha}$ and $\overline{\beta}$ to fix $D$, and obtain
    \begin{equation*}
        V = \bmat{0 & \sqrt{\alpha} & \sqrt{\beta} \\
                  -\sqrt{\alpha} & \overline{\beta} & - \sqrt{\overline{\alpha}} \sqrt{\overline{\beta}} \\
                  -\sqrt{\beta} & -\sqrt{\overline{\alpha}} \sqrt{\overline{\beta}} & \overline{\alpha}},
    \end{equation*}
    where $\sqrt{\alpha} = \overline{\sqrt{\overline{\alpha}}}$ and $\sqrt{\beta} = \overline{\sqrt{\overline{\beta}}}$ are fixed by the choice of $\sqrt{\overline{\alpha}}$ and $\sqrt{\overline{\beta}}$.

    It may be possible to generalize further and consider $\alpha, \beta \in \bC \setminus \{0\}$ such that $|\alpha| + |\beta| < 1$, however there are several choices for $d_{11}$ and $d_{22}$ already which makes the rest of the calculation quite tricky to keep track.
\end{example}

\section{The Non-commutative Case}\label{sec:zeros.hol.func.unit.Free}

\subsection{Background}\label{subsec:NC.background}

Recall that a matrix unit ball $\bD_Q \in \bC^d$ is defined as
\begin{equation*}
    \bD_Q := \{ z \in \bC^d : \|Q(z)\|< 1 \},
\end{equation*}
where $Q$ is an $s \times r$ matrix whose entries are linear polynomials, and the norm $\|Q(z)\|$ is computed under the operator topology on $\cB(\bC^s,\bC^r)$. Clearly, $\bD_Q$ is a non-empty domain in $\bC^d$ with $0 \in \bD_Q$. Note that we recover $\bD^d$ and $\bB_d$ by taking $Q = \diag(z_1, \dots, z_d)$ and $Q = [z_1 \dots z_d]$ respectively. Such domains have been extensively studied in the context of interpolation problems, von Neumann's inequality, and realizations (see, for instance, \cite{Ambrozie-Timotin-vonNeumann-inequality, Ball-Bolotnikov-matrix-unit-ball-1, Ball-Bolotnikov-matrix-unit-ball-2}).

As noted in the introduction, Ambrozie and Timotin \cite{Ambrozie-Timotin-vonNeumann-inequality} introduced what we call the Schur--Agler class over $\bD_Q$, denoted by $\cS\cA(\bD_Q)$, which consists of all $f \in \operatorname{Hol}(\bD_Q)$ such that $\|f(T)\| \leq 1$ whenever $T = (T_1,\dots,T_d)$ is a tuple of operators $T_j \in \cB(\cH)$ such that
\begin{equation*}
    T_j T_k = T_k T_j, \; \forall 1 \leq j,k \leq d \qquad \AND \qquad \|Q(T)\| < 1.
\end{equation*}
Their main result then generalizes Agler's theorem to matrix unit balls and shows that $f \in \cS\cA(\bD_Q)$ if and only if it admits a realization formula of the the form:
\begin{equation}\label{eqn:scalar.D_Q.realization.formula}
    f(z) = A + B[I - (Q(z) \otimes I_\cH)D]^{-1}(Q(z) \otimes I_\cH) C,
\end{equation}
where $\cH$ is an auxillary Hilbert space with a unitary colligation
\begin{equation*}
    V = \bmat{A & B \\ C & D} : \bmat{\bC \\ \bC^s \otimes \cH} \to \bmat{\bC \\ \bC^r \otimes \cH}.
\end{equation*}
Based on the realization formula, we recover $\cS\cA(\bD^d)$ and $\cM(\bB_d)_1$ as the Schur--Agler classes of $\bD_Q = \bD^d$ and $\bD_Q = \bB_d$ respectively. Thus, our analysis from earlier suggests that zeros of a function $f$ as in \eqref{eqn:scalar.D_Q.realization.formula} must be connected to certain eigenvalues of $D^*$ that arise from the operator structure induced by $Q$.

With the theory of non-commutative (NC) functions, one is able to take this notion one step further. A particularly interesting work in this context is a paper of Ball, Marx and Vinnikov \cite{Ball-Marx-Vinnikov-NCTFR} which generalizes Ambrozie and Timotin's techniques to the setting of NC matrix unit balls. We follow closely the notation from \cite{Ball-Marx-Vinnikov-NCTFR} and provide basics of NC function theory for the uninitiated.

\subsubsection{\texorpdfstring{\textbf{NC universe}}{NC universe}}\label{subsubsec:NC.universe}

For $n,m \in \bN$, we write $M_{n \times m}$ for the space of all $n \times m$ matrices with complex entries. For any $d \in \bN$, we then define the \emph{NC universe} $\bM^d$ as the graded union
\begin{equation*}
    \bM^d := \bigsqcup_{n \in \bN} M_{n \times n} \otimes \bC^d,
\end{equation*}
consisting of $d$-tuples $X = (X_1, \dots, X_n)$ of matrices $X_j \in M_{n \times n}$ of fixed but arbitrary size $n \in \bN$.

One might wish to replace $\bC^d$ with any operator space and generalize the above definition. For Hilbert spaces $\cU$ and $\cV$, we define the corresponding \emph{NC operator space} $\cB(\cU,\cV)_{\nc}$ as
\begin{equation*}
    \cB(\cU,\cV)_{\nc} := \bigsqcup_{n \in \bN} \cB(\cU \otimes \bC^n, \cV \otimes \bC^n) = \bigsqcup_{n \in \bN} M_{n \times n} \otimes \cB(\cU, \cV).
\end{equation*}

\subsubsection{\texorpdfstring{\textbf{NC sets and domains}}{NC sets and domains}}\label{subsubsec:NC.set.and.domain}

For any $\Omega \subset \bM^d$ and $n \in \bN$, we say that $\Omega_n := \Omega \cap (M_{n \times n} \otimes \bC^d)$ is the \emph{$n^{\text{th}}$ level} of $\Omega$. $\Omega$ is said to be an \emph{NC set} if it is closed under direct sums, i.e.,
\begin{equation*}
    X \in \Omega_n, Y \in \Omega_m \implies X \oplus Y := \bmat{X & 0 \\ 0 & Y} \in \Omega_{n + m}.
\end{equation*} Each $M_{n \times n} \otimes \bC^d$ is equipped with the supremum norm:
\begin{equation*}
    \|X\|_\infty := \max_{1 \leq j \leq d} \|X_j\|,
\end{equation*}
and we can equip $\bM^d$ with the \emph{disjoint-union topology} (also called the \emph{finite open topology}), i.e., $\Omega \subset \bM^d$ is open if and only if $\Omega_n$ is open for all $n \in \bN$. If $\Omega$ is an NC set that is also open, we say that $\Omega$ is an \emph{NC domain}. For instance, the \emph{NC unit polydisk}
\begin{equation*}
    \fD^d := \{ X \in \bM^d : \|X\|_\infty < 1 \}
\end{equation*}
is an NC domain, and so is the \emph{NC unit row-ball}
\begin{equation*}
    \fB_d := \left\{ X \in \bM^d : \left\|\sum_{j = 1}^d X_j X_j^* \right\| < 1 \right\}.
\end{equation*}

\subsubsection{\texorpdfstring{\textbf{NC functions}}{NC functions}}\label{subsubsec:NC.func}

Let $\cU$ and $\cV$ be Hilbert spaces and let $\Omega \subset \bM^d$ be an NC domain. A map $f : \Omega \to \cB(\cU,\cV)_{\nc}$ is said to be an \emph{NC function} if
\begin{enumerate}
    \item $f$ is graded: $X \in \Omega_n \implies f(X) \in \cB(\cU, \cV) \otimes M_{n \times n}$,
    \item $f$ respects direct sums: $X, Y \in \Omega \implies f(X \oplus Y) = f(X) \oplus f(Y)$, and
    \item $f$ respects joint similarities: $X = (X_1, \dots, X_d) \in \Omega_n$, $S \in GL_n$ and $S^{-1}X S := (S^{-1} X_1 S, \dots, S^{-1} X_d S) \in \Omega_n \implies f(S^{-1} X S) = S^{-1} f(X) S$.
\end{enumerate}

A notable feature of free analysis is that a mild local boundedness condition on NC functions is sufficient to ensure holomorphicity (in an appropriate sense). In fact, since we will only deal with NC functions that are bounded over certain NC domains, it will automatically mean that the function is holomorphic on its domain! The interested reader is diverted to \cite{Vinnikov-Verbovetskyi-book} for more details.

\subsubsection{\texorpdfstring{\textbf{NC matrix unit balls}}{NC matrix unit balls}}\label{subsubsec:NC.mat.unit.ball}

Let $\cU$ and $\cV$ be finite-dimensional Hilbert spaces and let
\begin{equation*}
    \{Q_1,\dots,Q_d\} \subset \cB(\cU,\cV)
\end{equation*}
be linearly independent. Then, consider the \emph{linear NC matrix polynomial} $Q: \bM^d \to \cB(\cU,\cV)_{\nc}$ given by
\begin{equation*}
    Q(X) := \sum_{j = 1}^d Q_j \otimes X_j \foral X \in \bM^d.
\end{equation*}
The \emph{NC matrix unit ball} $\bD_Q$ corresponding to $Q$ is then defined as the NC set
\begin{equation}\label{eqn:NC.mat.unit.ball}
    \bD_Q := \{X \in \bM^d : \|Q(X)\| < 1 \}.
\end{equation}
Here, the norm of $Q(X)$ is computed in the operator topology of $M_{n \times n} \otimes \cB(\cU,\cV)$ if $X \in \Omega_n$. In this paper, we focus only on the case when $\cU, \cV$ are finite-dimensional. If
\begin{equation*}
    \dim \cU = r \qand \dim \cV = s,
\end{equation*}
then $Q$ is an $s \times r$ matrix with entries that are linear NC polynomials in $X_j$'s.

Note that $\fD^d$ and $\fB_d$ are obtained by choosing $Q_{\diag}(X) := \diag(X_1, \dots, X_d)$ and $Q_{\row}(X) := [X_1 \dots X_d]$ respectively. Showcasing a deviation from the commutative case, we also have the \emph{NC unit column ball} $\fC_d$ as the NC matrix unit ball given by $Q_{\col}(X) := \bmat{X_1 \\ \vdots \\ X_d}$.

Clearly, $\bD_Q$ is non-empty since $0 \in \bD_Q$, and that $r X \in \bD_Q$ for each $0 < r < 1$ and $X \in \bD_Q$. It is also straightforward to check that $\bD_Q$ is a bounded NC domain (i.e., $ \sup_{X \in \bD_Q} \|X\|_\infty < \infty$) that is \emph{uniformly open} (see \cite[Proposition 2.6]{Sampat-Shalit-JMAA}). Recall that the uniform open topology on $\bM^d$ is generated by basic NC open balls of the form:
\begin{equation*}
    B_{\nc}(X,r) := \bigsqcup_{m \in \bN} \left\{ Y \in M_{mn \times mn} \otimes \bC^d : \left\| Y - X^{(m)} \right\| < r \right\},
\end{equation*}
where $X \in M_{n \times n} \otimes \bC^d$, $X^{(m)} := \underbrace{X \oplus \dots \oplus X}_{m \text{ times}}$ and $r \in (0,\infty)$. Moreover, the topological boundary of $\bD_Q$ is given by
\begin{equation*}
    \partial \bD_Q = \{X \in \bM^d : \|Q(X)\| = 1 \}.
\end{equation*}

\subsection{NC Schur--Agler class and NC realizations}\label{subsec:bdd.NC.func.and.realization}

Let $\bD_Q$ be as in \eqref{eqn:NC.mat.unit.ball}. Define $H^\infty(\bD_Q)$ to be the space of all bounded NC functions on $\bD_Q$, i.e.,
\begin{equation*}
    H^\infty(\bD_Q) := \left\{ f : \bD_Q \to \bM^1 : f \text{ is NC and} \|f\|_\infty := \sup_{X \in \bD_Q} \|f(X)\| < \infty \right\}.
\end{equation*}
The \emph{NC Schur--Agler class} over $\bD_Q$ is then defined as its unit ball, i.e., $\cS\cA(\bD_Q) := H^\infty(\bD_Q)_1$. As noted earlier, any $f \in H^\infty(\bD_Q)$ is automatically NC holomorphic. In fact, \cite[Theorem 7.21]{Vinnikov-Verbovetskyi-book} shows that it is holomorphic with respect to the uniform open topology and, therefore, has a global NC power-series representation of the form
\begin{equation*}
    f(Z) = \sum_{\alpha \in \bF_d^+} c_\alpha Z^\alpha,
\end{equation*}
where $\bF_d^+$ is the free unital semigroup generated by the \emph{alphabet} $\{1, \dots, d\}$, and a \emph{word} $\alpha = \alpha_1 \dots \alpha_k \in \bF_d^+$ corresponds to an \emph{NC monomial} $Z^\alpha := Z_{\alpha_1} \dots Z_{\alpha_k}$. 

The main result of \cite{Ball-2006} generalizes the commutative version of the realization formula to general NC power-series that converge over $\bD_Q$. We use the following modification from \cite{Ball-Marx-Vinnikov-NCTFR} that works for functions in $H^\infty(\bD_Q)$ as we define. In particular, \cite[Remark 2.21 and Corollary 3.2]{Ball-Marx-Vinnikov-NCTFR} show that $f \in \cS\cA(\bD_Q)$ if and only if there is an auxillary Hilbert space $\cH$ and a unitary colligation
\begin{equation}\label{eqn:NC.colligation}
    V := \bmat{A & B \\ C & D} : \bmat{\bC \\ \cV \otimes \cH} \to \bmat{\bC \\ \cU \otimes \cH}
\end{equation}
such that
\begin{equation}\label{eqn:NC.realization.formula}
    f(X) = A^{(n)} + B^{(n)}[I - (Q(X) \otimes I_\cH) D^{(n)}]^{-1}(Q(X) \otimes I_\cH)C^{(n)}
\end{equation}
for all $X \in {\bD_Q}_n$ and $n \in \bN$.

\subsection{Zeros and NC \texorpdfstring{$Q$}{Q}-eigenvalues}\label{subsec:zeros.and.NC.Q-eigenvalues}

Unlike the commutative case, there are a couple of different notions of zeros for NC functions that appear naturally in different contexts. First, we have the \emph{exact zeros}, i.e., given an NC function $f$ on an NC domain $\Omega \subset \bM^d$ and $n \in \bN$, we say that $X \in \Omega_n$ is an exact zero of $f$ if $f(X) = 0_n$. Such zeros have appeared when considering the isomorphism problem of NC function algebras and their classification (see \cite{Salomon-Shalit-Shamovich-algebras-1, Sampat-Shalit-JFA}). Next, we have the \emph{determinantal zeros}, i.e, $X$ is a determinantal zero of $f$ if $\det f(X) = 0$. Such zeros appeared in \cite{Helton-Klep-Volcic-free-loci} in the context of factorization results for non-commutative polynomials. Somewhat recently this type of zeros appeared in \cite{Jury-Martin-Shamovic-Blaschke}, where the authors proved that any given function in the \emph{full Fock space} or the \emph{non-commutative Hardy space} have a version of the Blaschke--Singular Inner--Outer factorization -- just as in the classical case for $H^2(\bD)$, and that the Blaschke factor captures the zero variety of the given function completely. In this paper, we work with determinantal zeros but reinterpret them as follows.

Given any $f \in H^\infty(\bD_Q)$, the \emph{zero locus} of $f$ in $\bD_Q$ is defined as the NC set
\begin{equation*}
    \cZ_{\bD_Q}(f) := \bigsqcup_{n \in \bN} \{X \in {\bD_Q}_n : f(X)y = 0 \text{ for some } 0 \neq y \in \bC^n\}.
\end{equation*}

Since the zero loci consist of matrices, we must introduce yet another notion of `eigenvalues' to connect with the zeros.

\begin{definition}\label{def:NC.Q-eigenvalues}
    Let $T \in \cB(\cU \otimes \cH, \cV \otimes \cH)$ for some Hilbert spaces $\cH, \cU, \AND \cV$, and let $Q : \bM^d \to \cB(\cU, \cV)_{\nc}$ be an NC map. We say $\Lambda \in M_{n \times n} \otimes \bC^d$ is an \emph{NC $Q$-eigenvalue} (at level $n$) if there is a non-zero vector $\vec{v} \in \cU \otimes \cH \otimes \bC^n$ such that
    \begin{equation*}
        T^{(n)}\vec{v} = (Q(\Lambda) \otimes I_\cH) \vec{v}.
    \end{equation*}
    We denote the set of all NC Q-eigenvalues of $T$ at level $n$ as $\sigma^Q_p(T^{(n)})$ and the set of all NC Q-eigenvalues as
    \begin{equation*}
        \sigma^Q_p(T) := \bigsqcup_{n \in \bN} \sigma^Q_p(T^{(n)}).
    \end{equation*}
\end{definition}

It follows from the definition that $\sigma^Q_p(T)$ is an NC set. We record this below.

\begin{lemma}\label{lemma:nceigenvalue}
  $\sigma^Q_p(T)$ is an NC set.
\end{lemma}

With these definitions in place, we are ready to prove Theorem \ref{mainthm:zeros.NC.Schur--Agler.class.realization}.

\subsubsection*{\texorpdfstring{\textbf{Proof of Theorem \ref{mainthm:zeros.NC.Schur--Agler.class.realization}}}{Proof of Theorem C}}

Suppose $\Lambda \in \cZ_{\bD_Q}(f)_n$ for some $n \in \bN$, and let $0 \neq y \in \bC^n$ be such that $f(\Lambda)y = 0$. We introduce the vector
\begin{equation}\label{eqn:NC.zero.proof.1}
    {L_Q(\Lambda)} := [I - (Q(\Lambda) \otimes I_\cH)D^{(n)}]^{-1}(Q(\Lambda) \otimes I_\cH)C^{(n)}y \in \cV \otimes \cH \otimes \bC^n
\end{equation}
as in the commutative case. Now, we note that
\begin{equation*}
    \bmat{A^{(n)} & B^{(n)} \\ C^{(n)} & D^{(n)}} \bmat{y \\ L_Q(\Lambda)} = \bmat{0 \\ C^{(n)}y + D^{(n)} L_Q(\Lambda)}.
\end{equation*}
We then choose $\vec{v} := C^{(n)}y + D^{(n)}L_Q(\Lambda) \in \cU \otimes \cH \otimes \bC^n$ and apply $V^{(n)*}$ to both sides of \eqref{eqn:NC.zero.proof.1} to obtain
\begin{equation}
    \bmat{y \\ L_Q(\Lambda)} = \bmat{A^{(n)*} & C^{(n)*} \\ B^{(n)*} & D^{(n)*}} \bmat{0 \\ \vec{v}} = \bmat{C^{(n)*}\vec{v} \\ D^{(n)*}\vec{v}}.
\end{equation}
It follows from comparing the first block above that $\vec{v} \neq \vec{0}$. Thus, in order to conclude that $\Lambda \in \sigma^Q_p(D^{(n)*}) \cap {\bD_Q}_n$, it suffices to show that
\begin{equation*}
   L_Q(\Lambda) = (Q(\Lambda) \otimes I_\cH) \vec{v}.
\end{equation*}
The following calculation shows that this holds:
\begin{alignat}{2}\label{eqn:temp.1}
    && [I - (Q(\Lambda) \otimes I_\cH) D^{(n)}] L_Q(\Lambda) &= (Q(\Lambda) \otimes I_\cH) C^{(n)}y \nonumber\\
    \implies && L_Q(\Lambda) - (Q(\Lambda)\otimes I_\cH) D^{(n)} L_Q(\Lambda) &= (Q(\Lambda) \otimes I_\cH) C^{(n)}y \nonumber\\
    \implies && L_Q(\Lambda) - (Q(\Lambda) \otimes I_\cH) (\vec{v} - C^{(n)} y) &= (Q(\Lambda) \otimes I_\cH) C^{(n)}y \nonumber\\
    \implies && L_Q(\Lambda) &= (Q(\Lambda) \otimes I_\cH) (C^{(n)}y + \vec{v} - C^{(n)}y) \nonumber\\
    \implies && L_Q(\Lambda) &= (Q(\Lambda) \otimes I_\cH) \vec{v}.
\end{alignat}

Conversely, let $\Lambda \in \sigma^Q_p(D^{(n)*}) \cap {\bD_Q}_n$ for some $n \in \bN$ and let $\vec{v}^{\Lambda} \in \cU \otimes \cH \otimes \bC^n$ be a non-zero vector so that $D^{(n)*} \vec{v}^{\Lambda} = (Q(\Lambda) \otimes I_\cH) \vec{v}^\Lambda$.
Observe that
\begin{equation*}
    \bmat{A^{(n)*} & C^{(n)*} \\ B^{(n)*} & D^{(n)*}} \bmat{0 \\ \vec{v}^\Lambda} = \bmat{C^{(n)*}\vec{v}^{\Lambda} \\ (Q(\Lambda) \otimes I_\cH) \vec{v}^\Lambda}.
\end{equation*}
Applying $V$ to both side above, we get
\begin{align}\label{eqn:NC.zero.proof.3}
    \bmat{0 \\ \vec{v}^\Lambda} &= \bmat{A^{(n)} & B^{(n)} \\ C^{(n)} & D^{(n)}} \bmat{C^{(n)*}\vec{v}^{\Lambda} \\ (Q(\Lambda) \otimes I_\cH) \vec{v}^\Lambda}\nonumber \\
    &= \bmat{A^{(n)}(C^{(n)*}\vec{v}^\Lambda) + B^{(n)}(Q(\Lambda) \otimes I_\cH) \vec{v}^\Lambda \\ C^{(n)}(C^{(n)*}\vec{v}^\Lambda) + D^{(n)}(Q(\Lambda) \otimes I_\cH) \vec{v}^\Lambda}
\end{align}
As in the commutative case, we claim that $C^{(n)*} \vec{v}^\Lambda \neq 0$. Indeed, if $C^{(n)}\vec{v}^\Lambda = 0$ then comparing the second block in \eqref{eqn:NC.zero.proof.3} gives us
\begin{equation*}
    D^{(n)}(Q(\Lambda)\ \otimes I_\cH)\vec{v}^\Lambda = \vec{v}^\Lambda,
\end{equation*}
which cannot happen since $D^{(n)}$ is a contraction, $\Lambda \in \bD_Q$, and
\begin{equation*}
    \|(Q(\Lambda) \otimes I_\cH)\vec{v}^\Lambda\| \leq \|Q(\Lambda) \otimes I_\cH\| \|\vec{v}^\Lambda\| = \|Q(\Lambda)\| \|\vec{v}^\Lambda\| < \|\vec{v}^\Lambda\|.
\end{equation*}
Thus, $y := C^{(n)*}\vec{v} \neq 0$ and we can define $L_Q(\Lambda)$ as in \eqref{eqn:NC.zero.proof.1}. Comparing the second block in \eqref{eqn:NC.zero.proof.3} and performing a calculation similar to \eqref{eqn:temp.1}, we get
\begin{equation*}
    (Q(\Lambda) \otimes I_\cH) \vec{v}^\Lambda = L_Q(\Lambda)
\end{equation*}
Substituting this to the right side of the first block in \eqref{eqn:NC.zero.proof.3} and then comparing both these sides gives us $f(\Lambda)y = 0$. Thus, $\Lambda \in \cZ_{\bD_Q}(f)$ as required, which completes the proof. \hfill \qedsymbol

\begin{remark}\label{rem:commutative.D_Q.case}
    Note that the above proof works just as fine by taking $n = 1$ for the Schur--Agler class over the scalar matrix unit ball $\bD_Q \subset \bC^d$ using the realization formula given in \eqref{eqn:scalar.D_Q.realization.formula}. In this case, we only need the scalar level of the NC $Q$-eigenvalues which we denote by $\sigma^Q_p(T)$, and the standard zero set in $\bD_Q$ which we also denote by $\cZ_{\bD_Q}(f)$ as in the NC case. We record this below.
\end{remark}

\begin{theorem}\label{thm:zeros.commutative.D_Q.SA.class}
    Let $\bD_Q \subset \bC^d$ be a matrix unit ball, and let $f \in \cS\cA(\bD_Q)$ be a non-constant function that has a unitary realization as in \eqref{eqn:scalar.D_Q.realization.formula} with a colligation $V = \bsmallmat{A & B \\ C & D}$. Then, we have that
    \begin{equation*}
        \cZ_{\bD_Q}(f) = \sigma^Q_p(D^*) \cap \bD_Q.
    \end{equation*}
\end{theorem}
\section{Boundary Values and the Approximate Point Spectrum}\label{sec:boundary.values.approx.pt.spec}

Boundary values help determine the structure of functions in holomorphic function spaces. For instance, consider $H^\infty(\bD)$ and recall the theorem of Fatou which states that every $f \in H^\infty(\bD)$ has radial limits Lebesgue a.e. on $\bT$, i.e.,
\begin{equation*}
    \lim_{r \to 1} f(r \lambda)
\end{equation*}
exists for Lebesgue a.e. $\lambda \in \bT$. This extends to the entire class $H^p(\bD)$ (for $p \in (0, \infty]$) of Hardy spaces on $\bD$, since every function in $H^p(\bD)$ can be written as a ratio of two $H^\infty(\bD)$ functions. Then, it can be shown that every $f \in H^p(\bD)$ is uniquely determined by its boundary function $f^\dagger$ and that $f^\dagger \in L^p(\bT)$ if $p \geq 1$. This observation is linked to several important results such as Smirnov's factorization theorem, Beurling's theorem, etc. Boundary zeros of holomorphic functions play a pivotal role in problems related to cyclicity (see \cite{Brown-Shields}).

Moving to the multivariate case we branch into two directions. On the one hand, we have the unit ball $\bB_d$ for some $d > 1$. A result of Kor\'anyi \cite{Koranyi-TAMS-1669} generalizes Fatou's theorem and shows that every $f \in H^\infty(\bB_d)$ (and hence, every $f \in \cM(\bB_d)$) has radial limits a.e. on $\bS_d$. On the other hand, we have the unit polydisk $\bD^d$ for some $d > 1$. The situation is notably distinct from the $\bB_d$ case since we get radial limits a.e. along the Shilov boundary $\bT^d$ (also, the distinguished boundary) \cite[Theorem 2.3.1]{Rudin-Book-Polydisk}, which is a rather small piece of the topological boundary $\partial \bD^d$. See \cite[Chapter 5]{Agler-McCarthy-Young-Book} and \cite{Jury-Martin-Fatou} for another flavor and study of problems related to boundary values.

Our goal in this section is twofold. First, we show that the boundary zeros of functions in $\cS\cA(\bD_Q)$ are captured by a version of the approximate point spectrum of its realization operator. Here, we consider zeros along the topological boundary. Lastly, we show that if we consider just the Shilov boundary (to be made precise later) then we can somewhat capture the points along this boundary where the function has a boundary value (not necessarily a zero). We present the main results of this section in the language of NC function theory and mention their commutative counterparts as standalone results at the end.

\subsection{Boundary zeros of NC functions}\label{subsec:boundary.zeros.NC.func}

Let $\bD_Q \subset \bM^d$ be an NC matrix unit ball as in \eqref{eqn:NC.mat.unit.ball}. Recall from Section \ref{subsubsec:NC.mat.unit.ball} that $rX \in \bD_Q$ whenever $0 < r < 1$ and $X \in \bD_Q$, and that the topological boundary of $\bD_Q$ is given by
\begin{equation*}
    \partial \bD_Q := \{ X \in \bM^d : ||Q(X)\| = 1 \}.
\end{equation*}

\begin{definition}\label{def:boundary.zero.NC}
    We say that $\Lambda \in \partial {\bD_Q}_n$ for some $n \in \bN$ is a \emph{boundary zero} of $f$ if there exists $0 \neq y \in \bC^n$ such that
    \begin{equation*}
        \lim_{r \to 1} f(r \Lambda)y = 0.
    \end{equation*}
    The set of all boundary zeros of $f$ (at level $n$) is denoted by $\cZ_{\partial \bD_Q} (f)_n$, and we define the \emph{boundary zero locus} of $f$ as
    \begin{equation*}
        \cZ_{\partial \bD_Q}(f) := \bigsqcup_{n \in \bN} \cZ_{\partial \bD_Q}(f)_n.
    \end{equation*}
\end{definition}

It is straightforward to check that $\cZ_{\partial \bD_Q}(f)$ is also an NC set. We will connect boundary zero loci with the following spectral object.

\begin{definition}\label{def:NC.Q.approx.pt.spec}
    Let $T \in \cB(\cU \otimes \cH, \cV \otimes \cH)$ for some Hilbert spaces $\cH, \cU, \AND \cV$, and let $Q : \bM^d \to \cB(\cU, \cV)_{\nc}$ be an NC map. We say that $\Lambda \in M_{n \times n} \otimes \bC^d$ is an \emph{NC $Q$-approximate eigenvalue} of $T$ (at level $n$) if there exists a sequence of unit vectors $\vec{v}_k \in \cU \otimes \cH \otimes \bC^n$ such that
\begin{equation*}
    \lim_{k \to \infty} \|T^{(n)} \vec{v}_k - (Q(\Lambda) \otimes I_\cH) \vec{v}_k\| = 0.
\end{equation*}

The set of all NC $Q$-approximate eigenvalue at level $n \in \bN$ will be denoted by $\sigma^Q_{ap}(T^{(n)})$, and the \emph{NC $Q$-approximate spectrum} is
\begin{equation*}
    \sigma^Q_{ap}(T) := \bigsqcup_{n \in \bN} \sigma^Q_{ap}(T^{(n)}).
\end{equation*}
\end{definition}

The next lemma is straight-forward to check from the definition.

\begin{lemma}\label{lemma:NC.approx.eigenvals.is.NC.set}
    $\sigma^Q_{ap}(T)$ is an NC set and $\sigma^Q_p(T) \subseteq \sigma^Q_{ap}(T)$.
\end{lemma}


\subsubsection*{\texorpdfstring{\textbf{Proof of Theorem \ref{mainthm:boundary.zeros.NC}}}{Proof of Theorem D}}

Suppose $\Lambda \in \cZ_{\partial \bD_Q}(f)_n$ for some $n \in \bN$ with $0 \neq y \in \bC^n$ as in Definition \ref{def:boundary.zero.NC}, and let $\{r_k\}_{k \in \bN} \subset (0,1)$ be such that $\lim_{k \to \infty} r_k = 1$. We introduce the sequence of vectors
\begin{equation*}
   L_k(\Lambda) := [I - (Q(r_k \Lambda) \otimes I_\cH) D^{(n)}]^{-1} (Q(r_k \Lambda) \otimes I_\cH) C^{(n)}y \in \cV \otimes \cH \otimes \bC^n
\end{equation*}
and note for each $k \in \bN$ that
\begin{equation}\label{eqn:NC.boundary.zero.approx.pt.spec.proof.1}
   \bmat{A^{(n)} & B^{(n)} \\ C^{(n)} & D^{(n)}} \bmat{y \\ L_k(\Lambda)} = \bmat{f(r_k \Lambda) y \\ C^{(n)}y + D^{(n)}L_k(\Lambda)}.
\end{equation}
Define $\vec{v}_k := C^{(n)}y + D^{(n)} L_k(\Lambda)$ and note using a calculation similar to \eqref{eqn:temp.1} that
\begin{equation}\label{eqn:NC.boundary.zero.approx.pt.spec.proof.1.5}
    L_k(\Lambda) = (Q(r_k \Lambda) \otimes I_\cH)\vec{v}_k \foral k \in \bN.
\end{equation}
Applying $V^*$ to both sides of \eqref{eqn:NC.boundary.zero.approx.pt.spec.proof.1} and using \eqref{eqn:NC.boundary.zero.approx.pt.spec.proof.1.5} gives us
\begin{equation}\label{eqn:NC.boundary.zero.approx.pt.spec.proof.3}
    \bmat{y \\ (Q(r_k \Lambda) \otimes I_\cH) \vec{v}_k} = \bmat{A^{(n)*}f(r_k \Lambda)y + C^{(n)*}\vec{v}_k \\ B^{(n)*}f(r_k \Lambda)y + D^{(n)*}\vec{v}_k}.
\end{equation}
Comparing the first block in \eqref{eqn:NC.boundary.zero.approx.pt.spec.proof.3} and taking the limit as $k \to \infty$ we get
\begin{equation*}
    \lim_{k \to \infty} C^{(n)*}\vec{v}_k = y - \lim_{k \to \infty} A^{(n)*} f(r_k \Lambda)y = y \neq 0.
\end{equation*}

This means that $\{\|\vec{v}_k\|\}_{k \in \bN}$ is bounded below and $\big\{\vec{v}_k/\|\vec{v}_k\|\big\}_{k \in \bN}$ is a sequence of unit vectors. Comparing the second block in \eqref{eqn:NC.boundary.zero.approx.pt.spec.proof.3} we get
\begin{align}\label{eqn:temp.2}
    \begin{split}
        \|D^{(n)*}\vec{v}_k - (Q(\Lambda) \otimes I_\cH) \vec{v}_k\| &\leq \|D^{(n)*}\vec{v}_k - (Q(r_k \Lambda) \otimes I_\cH) \vec{v}_k\| \\ & \qquad + \|(Q(r_k \Lambda) \otimes I_\cH) \vec{v}_k - (Q(\Lambda) \otimes I_\cH) \vec{v}_k\|
        \end{split} \nonumber\\
        &= \|B^{(n)*} f(r_k \Lambda)y\| + (1 - r_k)\|(Q(\Lambda) \otimes I_\cH) \vec{v}_k\| \nonumber\\
        &\leq \|f(r_k \Lambda) y\| + (1 - r_k) \|\vec{v}_k\|.
\end{align}
Since $\|\vec{v}_k\|$ is bounded below, there exists $\delta > 0$ such that $ \inf_k \|\vec{v}_k\| \geq \delta$ and thus
\begin{equation*}
    \limsup_{k \to \infty} \frac{\|f(r_k \Lambda) y\|}{\|\vec{v}_k\|} \leq \lim_{k \to \infty}\|f(r_k \Lambda)y\| \limsup_{k \to \infty} \frac{1}{\|\vec{v}_k\|} \leq \delta \lim_{k \to \infty} \|f(r_k \Lambda) y\| = 0.
    \end{equation*}
Dividing both sides in \eqref{eqn:temp.2} by $\|\vec{v}_k\|$ and letting $k \to \infty$ shows $\Lambda \in \sigma^Q_{ap}(D^*)$. \hfill \qedsymbol

\subsection{Boundary values of NC functions}\label{subsec:boundary.val.NC.func}

\begin{definition} \label{def:boundaryvalue;nc}
    An NC function $f \in \cS\cA(\bD_Q)$ is said to have a \emph{boundary value} at $\Lambda \in \partial {\bD_Q}_n$ for some $n \in \bN$ if $\lim_{r \to 1} f(r \Lambda)$ exists. In this case, we denote the boundary value of $f$ at $\Lambda$ by
    \begin{equation*}
        f^\dagger(\Lambda) := \lim_{r \to 1} f(r \Lambda)
    \end{equation*}
    and say that $\Lambda$ is a \emph{point of boundary value} for $f$. The map $f^\dagger : BP(f) \to \bM^1$ is called the \emph{boundary function} of $f$. The set of all points of boundary values for $f$ will be denoted by $BP(f)$. Lastly, for each $t \in (0,1]$ we define
    \begin{align}
        BP(f,t) &:= \bigsqcup_{n \in \bN} \{ \Lambda \in BP(f)_n : \|f^\dagger(\Lambda)y\| < t \|y\| \text{ for some } 0 \neq y \in \bC^n \},\label{eqn:def.BP(f,1)}\\
        BP^*(f,t) &:= \bigsqcup_{n \in \bN} \{ \Lambda \in BP(f)_n : \|y^* f^\dagger(\Lambda)\| < t \|y^*\| \text{ for some } 0 \neq y \in \bC^n \}\label{eqn:def.BP*(f,1)}
    \end{align}
    where, for a vector $y \in \bC^n$, $y^*$ is the mapping $v \mapsto \ip{v,y}$ which can be identified with $y$ via the Riesz representation theorem.
\end{definition}

We will only focus on $BP(f,t)$ and $BP^*(f,t)$ for the case $t = 1$ in this paper. Note that $BP(f,1)$ consists of all $\Lambda \in BP(f)$ such that $f^\dagger(\Lambda)$ is not an isometry, and $BP^*(f,1)$ consists of those points for which $f^\dagger(\Lambda)$ is not a coisometry.

\begin{proposition}\label{prop:boundary.func.and.value.property}
    Let $f \in \cS\cA(\bD_Q)$ be given. Then, the following must hold:
    \begin{enumerate}
        \item $BP(f)$, $BP(f,t)$ and $BP^*(f,t)$ are NC sets for all $t \in (0,1]$.
        \item If $\Lambda \in {BP(f)}_n$ and $S \in GL_n$ for some $n \in \bN$ are such that $S^{-1} \Lambda S \in \partial \bD_Q$, then $S^{-1} \Lambda S \in BP(f)$ and $f^\dagger(S^{-1}\Lambda S) = S^{-1} f^\dagger(\Lambda) S$.
        \item $f^\dagger : BP(f) \to \bM^1$ is an NC function.
    \end{enumerate}
\end{proposition}

\begin{proof}
    Let $f$ be as in the hypothesis.
    \begin{enumerate}
        \item $BP(f)$ is an NC set, since $\Lambda_1, \Lambda_2 \in BP(f)$ implies
        \begin{equation}\label{eqn:boundary.val.prop.proof.1}
            f^\dagger\Big( \bmat{\Lambda_1 & 0 \\ 0 & \Lambda_2} \Big) = \lim_{r \to 1} \bmat{f(r\Lambda_1) & 0 \\ 0 & f(r \Lambda_2)} = \bmat{f^\dagger(\Lambda_1) & 0 \\ 0 & f^\dagger(\Lambda_2)}.
        \end{equation}
        Now, fix $t \in (0,1]$ and let $\Lambda_1 \in BP(f,t)_{n_1}, \, \Lambda_2 \in BP(f,t)_{n_2}$. If $0 \neq y_j \in \bC^{n_j}$ ($j = 1, 2$) are such that $\|f^\dagger(\Lambda_j) y_j \| < t \|y_j\|$, then note that, for $0 \neq y = \bmat{y_1 & y_2}^t \in \bC^{n_1 + n_2}$, we have
        \begin{align*}
            \Big\|f^\dagger\Big(\bmat{\Lambda_1 & 0 \\ 0 & \Lambda_2}\Big) \bmat{y_1 \\ y_2}\Big\|^2 &= \lim_{r \to 1} \Big\|\bmat{f(r \Lambda_1) & 0 \\ 0 & f(r \Lambda_2)} \bmat{y_1 \\ y_2}\Big\|^2 \\
            &= \|f^\dagger(\Lambda)y_1\|^2 + \|f^\dagger(\Gamma)y_2\|^2 \\
            &< t^2 \|y\|^2.
        \end{align*}
        Thus, $BP(f,t)$ is an NC set, and the proof for $BP^*(f,t)$ is similar.
        \item If $\Lambda \in BP(f)_n$ and $S \in GL_n$ are as in the hypothesis, then it is straightforward to check that
        \begin{equation}\label{eqn:boundary.val.prop.proof.2}
            \lim_{r \to 1} f(r S^{-1}\Lambda S) = S^{-1}\lim_{r \to 1} f(r \Lambda) S = S^{-1} f^\dagger(\Lambda)S.
        \end{equation}
        \item This follows immediately from \eqref{eqn:boundary.val.prop.proof.1} and \eqref{eqn:boundary.val.prop.proof.2}. \qedhere
    \end{enumerate}
\end{proof}


\begin{definition}\label{def:isometric.boundary.pt.NC}
    Let $\bD_Q \subset \bM^d$ be an NC matrix unit ball for some $d \in \bN$. We define the \emph{isometric} and the \emph{coisometric portions} of the boundary of $\bD_Q$ as
    \begin{align*}
        \partial_{\iso}\bD_Q &:= \{ \Lambda \in \bM^d : Q(\Lambda)^* Q(\Lambda) = I_\cU \}, \AND\\
        \partial_{\coiso}\bD_Q &:= \{ \Lambda \in \bM^d : Q(\Lambda) Q(\Lambda)^* = I_\cV \}
    \end{align*}
    respectively. If $\cU \cong \cV$, we also define the \emph{unitary portion} of the boundary as
    \begin{equation*}
        \partial_{\uni}\bD_Q = \partial_{\iso} \bD_Q \cap \partial_{\coiso} \bD_Q.
    \end{equation*}
    Lastly, the \emph{essential boundary} of $\bD_Q$ is defined as $\partial_e \bD_Q := \partial_{\iso} \bD_Q \cup \partial_{\coiso} \bD_Q$.
\end{definition}

\begin{remark}\label{rem:Shilov.vs.essential.boundary}
    It is shown in \cite[Example 1.5.51]{Upmeier-Book-Shilov} (see also \cite[Section 2.3]{Popa-Vinnikov}) that the Shilov boundary of
\begin{enumerate}
    \item the NC unit polydisk $\fD^d$ is $\partial_{\uni}(\fD^d)$,
    \item the NC unit row ball $\fB_d$ is $\partial_{\coiso}(\fB_d)$, and
    \item the NC unit column ball $\fC_d$ is $\partial_{\iso}(\fC_d)$.
\end{enumerate}
Moreover, a quick dimensional analysis shows that if $\dim \cU \neq \dim \cV$ then one of $\partial_{\iso}(\bD_Q)$ or $\partial_{\coiso}(\bD_Q)$ must always be empty. It follows that the Shilov boundaries of $\fB_d$ and $\fC_d$ coincide with their essential boundaries, and the Shilov boundary of $\fD^d$ is strictly contained in its essential boundary. At the scalar level (i.e., at level $n = 1$), however, we note that $\partial_e(\fD^d)_1 = \bT^d$ -- the Shilov boundary of $\bD^d$.
\end{remark}


\subsubsection*{\texorpdfstring{\textbf{Proof of Theorem \ref{mainthm:NC.boundary.val.approx.pt.spec}}}{Proof of Theorem E}}

Suppose $\Lambda \in BP(f,1)_n \cap \partial_{\iso}(\bD_Q)$ for some $n \in \bN$ with $0 \neq y \in \bC^n$ as in \eqref{eqn:def.BP(f,1)}, and let $\{r_k\}_{k \in \bN} \subset (0,1)$ be such that $\lim_{k \to \infty} r_k = 1$. As in the proof of Theorem \ref{mainthm:boundary.zeros.NC}, we introduce the sequence of vectors
\begin{equation*}
    L_k(\Lambda) := [I - (Q(r_k \Lambda) \otimes I_\cH) D^{(n)}]^{-1} (Q(r_k \Lambda) \otimes I_\cH) C^{(n)}y \in \cV \otimes \cH \otimes \bC^n
\end{equation*}
and note for each $k \in \bN$ that
\begin{equation}\label{eqn:NC.boundary.val.approx.pt.spec.proof.1}
    \bmat{A^{(n)} & B^{(n)} \\ C^{(n)} & D^{(n)}} \bmat{y \\ L_k(\Lambda)} = \bmat{f(r_k \Lambda) y \\ C^{(n)}y + D^{(n)}L_k(\Lambda)}.
\end{equation}
Define $\vec{v}_k := C^{(n)}y + D^{(n)} L_k(\Lambda)$ and note using a calculation similar to \eqref{eqn:temp.1} that
\begin{equation}\label{eqn:NC.boundary.val.approx.pt.spec.proof.1.5}
    L_k(\Lambda) = (Q(r_k \Lambda) \otimes I_\cH)\vec{v}_k.
\end{equation}
Since $V = \bsmallmat{A & B \\ C & D}$ is a unitary, we can compare the norms of the two vectors in \eqref{eqn:NC.boundary.val.approx.pt.spec.proof.1} and use the above equality to record the following relationship for later use:
\begin{equation}\label{eqn:NC.boundary.val.approx.pt.spec.proof.2}
    \|f(r_k \Lambda)y\|^2 + \|\vec{v}_k\|^2 = \|y\|^2 + \|(Q(r_k \Lambda) \otimes I_\cH) \vec{v}_k\|^2
\end{equation}
Applying $V^*$ to both sides of \eqref{eqn:NC.boundary.val.approx.pt.spec.proof.1} and using \eqref{eqn:NC.boundary.val.approx.pt.spec.proof.1.5} gives us
\begin{equation}\label{eqn:NC.boundary.val.approx.pt.spec.proof.3}
    \bmat{y \\ (Q(r_k \Lambda) \otimes I_\cH) \vec{v}_k} = \bmat{A^{(n)*}f(r_k \Lambda)y + C^{(n)*}\vec{v}_k \\ B^{(n)*}f(r_k \Lambda)y + D^{(n)*}\vec{v}_k}.
\end{equation}
Comparing the first block in \eqref{eqn:NC.boundary.val.approx.pt.spec.proof.3} and taking the limit as $k \to \infty$ we get
\begin{equation*}
    \lim_{k \to \infty} C^{(n)*}\vec{v}_k = y - A^{(n)*} f^\dagger(\Lambda)y = [I - A^{(n)*}f^\dagger(\Lambda)] y \neq 0
\end{equation*}
(since $\|A\| = \|f(0)\| < 1$ and $\|f^\dagger(\Lambda)\| \leq 1$ imply that $I - A^{(n)*}f^\dagger(\Lambda)$ is invertible).

This means that $\{\|\vec{v}_k\|\}_{k \in \bN}$ is bounded below and $\big\{\vec{v}_k/\|\vec{v}_k\|\big\}_{k \in \bN}$ is a sequence of unit vectors. Comparing the second block in \eqref{eqn:NC.boundary.val.approx.pt.spec.proof.3} we get
\begin{align*}
    \begin{split}
        \|D^{(n)*}\vec{v}_k - (Q(\Lambda) \otimes I_\cH) \vec{v}_k\| &\leq \|D^{(n)*}\vec{v}_k - (Q(r_k \Lambda) \otimes I_\cH) \vec{v}_k\| \\ & \qquad + \|(Q(r_k \Lambda) \otimes I_\cH) \vec{v}_k - (Q(\Lambda) \otimes I_\cH) \vec{v}_k\|
    \end{split} \\
    &= \|B^{(n)*} f(r_k \Lambda)y\| + (1 - r_k)\|(Q(\Lambda) \otimes I_\cH) \vec{v}_k\| \\
    &\leq \|f(r_k \Lambda) y\| + (1 - r_k) \|\vec{v}_k\|,
\end{align*}
for all $k \in \bN$. To complete the proof, it then suffices to show that
\begin{equation*}
    \lim_{k \to \infty} \frac{\|f(r_k \Lambda) y\|}{\| \vec{v}_k \|} = 0.
\end{equation*}
Consequently, it suffices to show that
\begin{equation*}
    \limsup_{k \to \infty} \frac{1}{\|\vec{v}_k\|} = 0,
\end{equation*}
since the boundary value of $f$ at $\Lambda$ exists. To this end, we finally use the fact that $\Lambda \in \partial_{\iso}(\bD_Q)$ and that $Q$ is a matrix of linear NC polynomials so that
\begin{align*}
    \|(Q(r_k \Lambda)\otimes I_\cH)\vec{v}_k\| &= r_k \|(Q(\Lambda) \otimes I_\cH) \vec{v}_k\| \\
    &= r_k \ip{(Q(\Lambda)^*Q(\Lambda) \otimes I_\cH)\vec{v}_k, \vec{v}_k} \\
    &= r_k \|\vec{v}_k\|
\end{align*}
for all $k \in \bN$. Substituting this back into \eqref{eqn:NC.boundary.val.approx.pt.spec.proof.2}, we get
\begin{alignat*}{2}
    &&\|f(r_k \Lambda) y\|^2 + \|\vec{v}_k\|^2 &= \|y\|^2 + r_k^2 \|\vec{v}_k\|^2 \\
    \implies && (1 - r_k^2)\|\vec{v}_k\|^2 &= \|y\|^2 - \|f(r_k \Lambda)y\|^2 \\
    \implies && \frac{1}{\|\vec{v}_k\|^2} &= \frac{1 - r_k^2}{\|y\|^2 - \|f(r_k \Lambda)y\|^2}.
\end{alignat*}
Since the denominator in the final inequality above is bounded below (using the fact that $\Lambda \in BP(f,1)$), it follows that
\begin{equation*}
    \limsup_{k \to \infty} \frac{1}{\|\vec{v}_k\|} = 0.
\end{equation*}

Now, suppose $\Lambda \in BP^*(f,1) \cap \partial_{\coiso}(\bD_Q)$ with $0 \neq y \in \bC^n$ as in \eqref{eqn:def.BP*(f,1)}, and let $\{r_k\}_{k \in \bN } \subset (0,1)$ be a sequence such that $\lim_{k \to \infty} r_k = 1$ as before. This time, we introduce the sequence of maps
\begin{equation*}
    R_k(\Lambda) := B^{(n)}[I - (Q(r_k \Lambda) \otimes I_\cH)D^{(n)}]^{-1}(Q(r_k \Lambda) \otimes I_\cH)
\end{equation*}
for each $k \in \bN$, and note that
\begin{equation*}
    \bmat{y^* & R_k(\Lambda)}\bmat{A^{(n)} & B^{(n)} \\ C^{(n)} & D^{(n)}} = \bmat{y^* f(r_k \Lambda) & y^*B^{(n)} + R_k(\Lambda)D^{(n)}}.
\end{equation*}
Then, similar to \eqref{eqn:NC.boundary.val.approx.pt.spec.proof.1.5}, we get the relationship
\begin{equation*}
    R_k(\Lambda) = (Q(r_k \Lambda) \otimes I_\cH) \vec{v}_k^*
\end{equation*}
for vectors $\vec{v}_k$ identified with $\vec{v}_k^* := y^* B^{(n)} + R_k(\Lambda)D^{(n)}$ for all $k \in \bN$. The rest of the argument is identical to the case above. This completes the proof. \hfill \qedsymbol

\subsection{Boundary values of commutative functions}\label{subsec:boundary.val.comm.func}

Note that a matrix unit ball $\bD_Q \subset \bC^d$ is simply the first level of an NC matrix unit ball. Therefore, the definitions of boundary zeros and the approximate point spectrum can be defined analogously as in the NC setup by taking $n = 1$ in Definitions \ref{def:boundary.zero.NC} and \ref{def:NC.Q.approx.pt.spec}. The following result then follows from the proof of Theorem \ref{mainthm:boundary.zeros.NC} by taking $n = 1$.

\begin{theorem}\label{thm:boundary.zero.commutative}
    Let $\bD_Q \subset \bC^d$ be a matrix unit ball, and let $f \in \cS\cA(\bD_Q)$ be a non-constant function admitting a unitary realization as in \eqref{eqn:scalar.D_Q.realization.formula} with a colligation $V = \bsmallmat{A & B \\ C & D}$. Then,
   $\cZ_{\partial \bD_Q(f)} \subseteq \sigma^Q_{ap}(D^*).$
   
\end{theorem}

In terms of boundary values, note that for any $f \in \cS\cA(\bD_Q)$ in the commutative Schur--Agler class we have
\begin{equation*}
    BP(f,1) = \{\lambda \in BP(f) : |f^\dagger(\lambda)| < 1 \} = BP^*(f,1).
\end{equation*}
Thus, if we define $\partial_{\iso} \bD_Q, \partial_{\coiso} \bD_Q, \partial_e \bD_Q,$ and $\sigma_{ap}^Q(T)$ in the appropriate fashion, we arrive at the following version of Theorem \ref{mainthm:NC.boundary.val.approx.pt.spec} for the commutative case.

\begin{theorem}\label{thm:boundary.val.commutative}
    If $f \in \cS\cA(\bD_Q)$ is as in the hypothesis of Theorem \ref{thm:boundary.zero.commutative}, then
    \begin{equation*}
        BP(f,1) \cap \partial_e \bD_Q \subseteq \sigma^Q_{ap}(D^*).
    \end{equation*}
\end{theorem}

Recall from Remark \ref{rem:Shilov.vs.essential.boundary} that $\partial_e \bD^d = \bT^d$ and $\partial_e \bB_d = \bS_d$ for each $d$. Also observe that for any $f \in \cS(\bD)$, $\sigma^Q_{ap}(D^*)$ is simply the approximate point spectrum of $D^*$.



\begin{example}\label{example:general.boundary.val}
    Let $f \in \cS\cA(\bD^d)$ be a rational inner function. We saw in Example
    \ref{example:general.zeros} that it has a finite dimensional unitary realization given by $V = \bsmallmat{A & B \\ C & D} \in \cB(\bC \oplus \bC^N)$ for some $N \in \bN$. It was also noted that
    \begin{equation*}
        \sigma_p^{\diag}(D^*) = \{ \lambda \in \bC^d : \det (D^* - \Delta(\lambda)) = 0 \} =: V(\det (D^* - \Delta(z))).
    \end{equation*}
    Now, if $\lambda \not \in V(\det(D^* - \Delta(\lambda)))$ then $D^* - \Delta(\lambda)$ is invertible, and hence there is no sequence of unit vectors $v_k \in \bC^d$ with $\|(D^* - \Delta(\lambda))v_k\| \to 0$. Consequently,
    \begin{equation}\label{eqn:temp}
        \sigma^{\diag}_{ap}(D^*) = \sigma^{\diag}_p(D^*) = \{\lambda \in \bC^d : \det(D^* - \Delta(\lambda)) = 0\}.
    \end{equation}
\end{example}

\begin{example}\label{example:weird.boundary.val}
    Let $\alpha,\beta \in \bC \setminus \{0\}$ be such that $|\alpha| + |\beta| = 1$, and let
    \begin{equation*}
        f(z,w) = f_{\alpha, \beta}(z,w) = \frac{zw - \alpha z - \beta w}{1 - \overline{\beta} z - \overline{\alpha} w} \in \cS(\bD^2)
    \end{equation*}
    be as in Example \ref{example:weird.realization}. We noted that $f$ has a unitary realization with a colligation $V = \bsmallmat{A & B \\ C & D}$ such that
    \begin{equation*}
        D^* = \bmat{\overline{\beta} & - \sqrt{\overline{\alpha}}\sqrt{\overline{\beta}} \\ -\sqrt{\overline{\alpha}}\sqrt{\overline{\beta}} & \overline{\alpha}}.
    \end{equation*}
    Using \eqref{eqn:temp}, we get
    $$
    \sigma^{\diag}_{ap}(D^*) = V(zw - \alpha z - \beta w) \supset \cZ_{\overline{\bD^2}}(f).
    $$
    It is also easy to verify that $BP(f,1) = \emptyset$ by noting that
    \begin{equation*}
        |\lambda \mu - \alpha \lambda - \beta \mu| = |1 - \overline{\beta} \lambda - \overline{\alpha} \mu| \foral (\lambda,\mu) \in \bT^2,
    \end{equation*}
    and that $f$ has a non-removable singularity at $\big(\tfrac{|\beta|}{\overline{\beta}},\tfrac{|\alpha|}{\overline{\alpha}}\big) \in V(zw - \alpha z - \beta w) \cap \bT^2$. We therefore conclude that
    \begin{equation*}
        \sigma^{\diag}_{ap}(D^*) \cap \overline{\bD^2} = \cZ_{\overline{\bD^2}}(f) \cup \{\big(\tfrac{|\beta|}{\overline{\beta}},\tfrac{|\alpha|}{\overline{\alpha}}\big)\} = \sigma^{\diag}_p(D^*) \cap \overline{\bD^2}.
    \end{equation*}
\end{example}

\section*{Acknowledgments}
The authors are indebted to Joseph A. Ball, Robert T.W. Martin, John E. McCarthy, and James E. Pascoe for several helpful discussions.


\bibliographystyle{plain}
\bibliography{references}

\end{document}